\documentclass[10pt]{amsart}
\usepackage{amssymb,amsfonts,lineno,amsthm,amscd,stmaryrd}
\usepackage[leqno]{amsmath}
\usepackage[numbers]{natbib}

\usepackage{tikz}

\theoremstyle{plain}
\newtheorem{thm}{Theorem}
\newtheorem{cor}{Corollary}
\newtheorem{lem}[cor]{Lemma}
\newtheorem{prop}[cor]{Proposition}
\theoremstyle{definition}

\newtheorem{remark}[cor]{Remark}

\numberwithin{cor}{section}
\numberwithin{equation}{section}

\DeclareMathOperator{\trace}{trace}

\DeclareMathOperator{\dist}{dist}
\DeclareMathOperator{\USC}{USC}
\DeclareMathOperator{\LSC}{LSC}

\newcommand{\C}{\ensuremath{\mathbb{C}}}
\newcommand{\R}{\ensuremath{\mathbb{R}}}
\newcommand{\rn}{\R^n}

\newcommand{\iden}{\ensuremath{I}}
\newcommand{\ep}{\varepsilon}
\newcommand{\Sy}{\ensuremath{\mathcal{S}_n}}
\newcommand{\sph}{\ensuremath{S^{n-1}}}
\newcommand{\co}[1]{{\mathcal{C}}_{#1}}
\newcommand{\coo}{{\mathcal{C}}_{\omega}}
\newcommand{\cob}[1]{\overline{\mathcal{C}}_{#1}\!\setminus \!\{0\}}
\newcommand{\cobo}{{\overline{\mathcal{C}}_{\omega}\!\setminus \!\{0\}}}
\newcommand{\hc}[1]{H_{#1}(\omega)}
\newcommand{\puccisub}[2]{\mathcal{P}^-_{#1,#2}}
\newcommand{\Puccisub}[2]{\mathcal{P}^+_{#1,#2}}
\newcommand{\PucciSub}{\Puccisub{\elp}{\Elp}}
\newcommand{\pucciSub}{\puccisub{\elp}{\Elp}}
\newcommand{\pucci}{\pucciSub}
\newcommand{\Pucci}{\PucciSub}
\newcommand{\Elp}{\Lambda}
\newcommand{\elp}{\lambda}

\newcommand{\pl}{Phragm\`en-Lindel\"of }

\begin{document}
\title[Singular solutions of fully nonlinear elliptic equations]{Singular solutions of fully nonlinear elliptic equations and applications}
\author[S. N. Armstrong]{Scott N. Armstrong}
\address{Department of Mathematics\\ The University of Chicago\\ 5734 S. University Avenue
Chicago, Illinois 60637.}
\email{armstrong@math.uchicago.edu}
\author[B. Sirakov]{Boyan Sirakov}
\address{UFR SEGMI, Universit\'e Paris 10\\
92001 Nanterre Cedex, France \\
and CAMS, EHESS \\
54 bd Raspail \\
75270 Paris Cedex 06, France}
\email{sirakov@ehess.fr}
\author[C. K. Smart]{Charles K. Smart}
\address{Courant Institute of Mathematical Sciences \\ 251 Mercer Street \\New York, NY 10012.}
\email{csmart@cims.nyu.edu}
\date{\today}

\begin{abstract}
We study the properties of solutions of fully nonlinear, positively homogeneous elliptic equations near boundary points of Lipschitz domains at which the solution may be singular. We show that these equations have two positive solutions in each cone of $\rn$, and the solutions are unique in an appropriate sense. We introduce a new method for analyzing the behavior of solutions near certain Lipschitz boundary points, which permits us to classify isolated boundary singularities of solutions which are bounded from either above or below. We also obtain a sharp Phragm\'en-Lindel\"of result as well as a principle of positive singularities in certain Lipschitz domains.
\end{abstract}

\maketitle

\section{Introduction and main results}

We study singular solutions of fully nonlinear, homogeneous, uniformly elliptic equations, such as
\begin{equation} \label{eq0}
F(D^2u) = 0,
\end{equation}
where $F$ is a Bellman-Issacs operator, that is, an inf-sup of linear uniformly elliptic operators. This paper is the second in a series which began in \cite{ASS}. In the latter work, we classified singular solutions of \eqref{eq0} in $\R^n\!\setminus \! \{ 0 \}$ (that are bounded on one side both at the origin and at infinity), and characterized the behavior of solutions of \eqref{eq0} at isolated singular points as well as their asymptotic behavior at infinity.

\medskip

Here we are interested in understanding how solutions of \eqref{eq0}, subject to Dirichlet boundary conditions, behave at isolated singularities on the boundary, as well as in the behavior ``at infinity" in unbounded domains which are not exterior domains (such as a half-space). We  prove a \emph{principle of positive singularities} for \eqref{eq0} in certain bounded Lipschitz domains, the first of its kind for fully nonlinear elliptic equations. Along the way, we discover new, \emph{sharp} maximum principles of \pl type, some of which are new even for linear equations, and we (sharply) generalize Hopf's lemma to domains with corners.

We proceed by first studying \eqref{eq0} in a cone of $\rn$ and classifying its solutions which do not change sign and are bounded either at the origin or at infinity. It turns out that for each conical domain, there are \emph{four} such solutions (up to normalization). These solutions are homogeneous: two have a positive order of homogeneity and the other two are of negative order. As we show, the orders of homogeneity of these four solutions provide a great deal of information about the general behavior of solutions of \eqref{eq0} near isolated singular boundary points and at infinity.

\medskip

Prior to this paper, the literature on singular solutions and domains for fully nonlinear equations is limited to the work of Miller \cite{Mi}, who long ago studied Pucci's extremal equations in symmetric cones via ODE methods. On the other hand, there is a vast literature of analogous results for linear and quasilinear equations (a good starting point to this theory is the recent book by Borsuk and Kondratiev \cite{BK}, and the references there). Our methods are completely different from the energy methods typically employed in these works, since fully nonlinear equations may be studied only by maximum principle arguments.

\medskip

We will consider a more general equation than \eqref{eq0}, namely
\begin{equation} \label{eq}
F(D^2u, Du, x)=0,
\end{equation}
where $F$ is a uniformly elliptic, positively homogeneous operator such that \eqref{eq} is invariant under dilations. The precise hypotheses on $F$ are \eqref{ellip}, \eqref{homogen}, \eqref{regxh}, and  \eqref{scaling} stated in Section~\ref{prelim}, which of course includes \eqref{eq0}. An example of such an equation is
\begin{equation*}
\Pucci(D^2u) + \mu|x|^{-1} |Du| = 0,
\end{equation*}
where $\mu \in \R$ and $\Pucci$ is the Pucci extremal operator (defined in the next section). We consider operators with gradient dependence not only for maximal generality, but mostly since doing so allows us to shorten many of our arguments by exploiting the fact that the class of these operators is invariant under the {inversion} $x\mapsto |x|^{-2} x$.

\medskip

Throughout this paper, we take $\omega$ to be a proper $C^2$-smooth subdomain of the unit sphere $S^{n-1}$ of $\rn$ in dimension $n\ge2$, and we define the conical domain
\begin{equation*}
\mathcal{C}_\omega : = \{ x \in \R^n : {|x|^{-1}}{x} \in \omega \}.
\end{equation*}
All differential inequalities in this article are to be understood in the viscosity sense.

\subsection{Existence and uniqueness of singular solutions in cones}

The first theorem gives the existence of positive singular solutions of \eqref{eq} in $\mathcal{C}_\omega$.

\begin{thm}\label{fundycones}
There exist unique constants
\begin{equation*}
\alpha^- < 0 < \alpha^+,
\end{equation*}
 depending on $F$ and $\omega$, for which the boundary value problem
\begin{equation}\label{eigpb}
\left\{ \begin{aligned}
& F(D^2\Psi,D\Psi,x) = 0 & \mbox{in} & \ \co\omega, \\
& \Psi = 0 & \mbox{on} & \ \partial \mathcal{C}_\omega\setminus \{ 0 \},
\end{aligned} \right.
\end{equation}
has two positive solutions $\Psi^+ \in C(\cobo)$ and $\Psi^- \in C(\overline{\mathcal{C}}_\omega)$ such that
\begin{equation*}
\Psi^\pm(x) = t^{\alpha^\pm} \Psi^\pm(tx) \quad \mbox{for all} \quad t > 0, \ x\in \R^n.
\end{equation*}
The numbers $\alpha^\pm$ are characterized by \eqref{ap} and \eqref{am} below.
\end{thm}


This theorem was proved in  \cite{Mi} in the special case that $F$ is a Pucci  extremal operator and $\omega$ is symmetric with respect to some diameter of the unit sphere of $\rn$. This situation is much simpler,  one may exploit the rotational invariance to reduce the problem to an ODE and then use a shooting argument to obtain smooth solutions.

\medskip

The next theorem states that $\Psi^+$ and $\Psi^-$ are unique in an appropriate sense. This result is new even for the Pucci extremal equations $\mathcal{P}^\pm_{\lambda,\Lambda}(D^2u) = 0$.

\begin{thm} \label{uniq}
Assume that $u\in C(\cobo)$ is nonnegative and satisfies
$$F(D^2u,Du,x) = 0 \quad \mbox{ in } \; \co\omega, $$ and $ u = 0$ on $\partial \mathcal{C}_\omega\setminus \{ 0 \}$.
Then
\begin{enumerate}
\item[(i)]  $\lim_{|x|\to \infty} |x|^{\alpha^-} u(x) = 0$ implies that $u \equiv t\Psi^+$ for some $t \geq 0$; and
\item[(ii)] $\lim_{|x|\to 0} |x|^{\alpha^+} u(x) = 0$ implies that $u \equiv t\Psi^-$ for some $t\geq 0$.
\end{enumerate}
In particular, if $u$ is bounded in $\{ |x| > 1 \}$, then $u$ is a nonnegative multiple of $\Psi^+$, while if $u$ is bounded near the origin, then $u$ is a nonnegative multiple of $\Psi^-$.
\end{thm}

Applying Theorem \ref{fundycones} to the operator
\begin{equation} \label{Ftilde}
\widetilde{F}(M,p, x) := -F(-M,-p, x),
\end{equation}
we see that \eqref{eigpb} has two unique \emph{negative} homogeneous solutions in $\coo$ as well.

\medskip

The proof of Theorem~\ref{fundycones} is inspired by the arguments in \cite{ASS} which have roots in the maximum principle approach to the principal eigenvalue theory. The idea is to write down an ``optimization" formula for $\alpha^+$ and $\alpha^-$, and then to use topological methods to find solutions.

To prove Theorem~\ref{uniq}, we use the global Harnack inequality, certain monotonicity properties implied by the comparison principle, and a blow-up argument.

\subsection{Behavior of solutions near ``singular" boundary points}\label{sect12}

In this subsection we show that the singular solutions $\Psi^\pm$ characterize the boundary singularities of solutions of the Dirichlet problem which are bounded on one side, and deduce a  Picard-Bouligand principle, as well as a generalized Hopf lemma for fully nonlinear equations.  Since we strive for clarity and readability, here we do not state our most general results, in particular with respect to the class of operators to which the following theorems apply. These can be found in  Section~\ref{ibsec1}.

In what follows, $\Omega\subset\rn$ is a domain such that $0\in \partial \Omega$ and $\partial \Omega$ has a conical-type corner at~$0$. Specifically, we assume that
\begin{equation} \label{zetaiden}
\left\{\begin{aligned}
& \mbox{there exists a} \  C^2\mbox{-diffeomorphism} \  \ \zeta: B_1 \to B_1 \ \ \mbox{such that} \\
& \zeta\left( \Omega \cap B_1 \right) = \co\omega \cap B_1, \quad \zeta(0) = 0 \quad \mbox{and} \quad D\zeta(0) = \iden.
\end{aligned} \right.
\end{equation}
Here $\iden$ is the identity matrix, we write $\zeta = (\zeta^1,\ldots,\zeta^n)$ and $D\zeta(x)$ denotes the matrix with entries $(\zeta_{x_j}^i(x))$. Observe that if $\partial \Omega$ is $C^2$ near $0$, then necessarily \eqref{zetaiden} holds and the cone $\co\omega$ is a half space.
 We emphasize that the general conical domain $\mathcal{C}_\omega$ is not the novelty here: both Theorems~\ref{ibs0} and \ref{poisson}, below, are new for smooth domains.

\begin{thm}\label{ibs0}
Assume $\Omega$ satisfies \eqref{zetaiden}. Let $u \in C\!\left( \bar\Omega  \!\setminus \! \{ 0 \} \right)$ be bounded below and
\begin{equation}\label{isol}
\left\{ \begin{aligned}
& F(D^2u,Du,x) = 0 & \mbox{in} & \ \Omega, \\
& u = 0 & \mbox{on} & \ \partial\Omega \cap B_1 \setminus \{ 0 \}.
\end{aligned} \right.
\end{equation}
Then precisely one of the following two alternatives holds:
\begin{enumerate}
\item[(i)] $u$ can be continuously extended to $\bar\Omega $ by setting $u(0)= 0$.
\item[(ii)] There exists $t> 0$ such that
\begin{equation*}
\lim_{\Omega \ni x\to 0}\frac{u(x)}{\Psi^+(\zeta(x))} =t.
\end{equation*}
\end{enumerate}
\end{thm}

\begin{remark} \label{remk}
In the case $u$ is positive near $0$ and (i) occurs, we can show using the same methods that there exists $t> 0$ such that
$
[u(y)]/[\Psi^-(\zeta(y))] \to  t
$ as $y\to 0$.
\end{remark}

An extension of the previous remark is the following generalized Hopf lemma.

\begin{thm}\label{exthopf} Let $\Omega$ be a domain such that $0\in \partial \Omega$ and \eqref{zetaiden} holds. If
\begin{equation*}
F(D^2u, Du, x) \ge 0 \quad \mbox{in} \quad \Omega, \qquad
 u \ge 0 \quad \mbox{in} \quad \Omega,\quad \mbox{and} \quad u(0) = 0,
\end{equation*}
 then either $u\equiv 0$ in $\Omega$ or $u>0$ in $\Omega$ and
$
\liminf_{x\to 0, x\in \Omega} \frac{u(x)}{\Psi^-(\zeta(x))}>0.
$
In particular, for any $ \omega^\prime\subset\subset\omega$ and $e\in \omega^\prime$ we have
\begin{equation} \label{hopfineq}
u(te)\ge ct^{-\alpha^-}\quad \mbox{as}\  t \to 0,
\end{equation}
where $c>0$ depends on $\lambda,\Lambda,\mu,\omega^\prime, \mbox{dist}(\omega^\prime,\partial \omega)$.
\end{thm}

The lower bound \eqref{hopfineq} is sharp, as the existence of $\Psi^-$ shows. If $F=F(D^2u)$ and $\Omega$ is smooth at the origin, it is obvious that after a rotation we have $\Psi^-(x) = x_n$ and  $-\alpha^-=1$, so in this case \eqref{hopfineq} is the conclusion of the usual Hopf lemma.

\medskip

Theorem \ref{ibs0} easily implies the following \emph{Picard-Bouligand principle} (or a \emph{principle of positive singularities}) which is a solvability and uniqueness theorem for fully nonlinear PDE in bounded domains. For linear equations such results, as well as discussion and references, can be found in Kemper \cite{Ke}, Ancona \cite{An}, Pinchover \cite{Pi}.

\begin{thm}\label{poisson}
Let $\Omega$ be a bounded Lipschitz domain such that $0\in \partial \Omega$ and \eqref{zetaiden} holds. Then the set of nonnegative solutions of the Dirichlet problem
\begin{equation} \label{dirpoisson}
\left\{ \begin{aligned}
& F(D^2u,Du,x) = 0 & \mbox{in} & \ \Omega, \\
& u = 0 & \mbox{on} & \ \partial\Omega \setminus \{ 0 \}.
\end{aligned} \right.
\end{equation}
is the half-line $\{ t u_0 : t \geq 0 \}$ for some positive solution $u_0 > 0$ of \eqref{dirpoisson}.
\end{thm}

By adopting the established terminology of the theory of harmonic functions, and its extensions to more general equations, we could restate Theorem \ref{poisson} by saying that {\it the $F$-Martin boundary of $\Omega$ coincides with $\partial \Omega$ provided any point $x_0\in\partial \Omega$ is such that a neighborhood of $x_0$ in $\partial \Omega$ is $C^2$-diffeomorphic to a cone in $\rn$}.
\medskip

While some of the ideas for the proof of Theorem~\ref{ibs0} are already present in  \cite{ASS} as well as in the earlier work of Labutin~\cite{L1}, the analysis here is much more challenging. The main difficulty comes when one smooths the boundary using the isomorphism $\zeta$ and discovers that the resulting equation is perturbed in its dependence on the second derivatives. Since $C^2$ estimates for solutions of general fully nonlinear equations are unavailable, to get a continuous dependence result we must first revisit the uniqueness machinery for viscosity solutions of second-order equations \cite{UG,IL}. We then combine these techniques with the global Harnack inequality for quotients and the global gradient H\"older estimates, to deduce a continuous dependence estimate in a somewhat unusual form. This estimate is then used to prove some almost-monotonicity properties of minima and maxima of ratios of solutions over annuli. A more detailed overview of the strategy for proving Theorem~\ref{ibs0} is given in Section~\ref{proofstrategy}.

\subsection{Maximum principles of \pl type}\label{PLsec}
To provide context for our next results, let us recall two very classical theorems of Phragm\`en and Lindel\"of. They proved, for any holomorphic function $f:\C \to \C$ which is bounded in the angle between two straight lines,
\begin{enumerate}
\item[(a)] if $|f|\le 1$ on these lines, then $|f|\le 1$ in the whole angle; and
\item[(b)] if $f(z)\to a$ as $|z|\to \infty$ along these lines then $f(z)\to a$ as $|z|\to \infty$, uniformly in the whole angle.
\end{enumerate}
These theorems may be formulated in terms of harmonic functions by applying the Cauchy-Riemann theorem to $\log f(z)$.

\medskip

Extensions and refinements of the statement (a) for subsolutions of linear elliptic equations in unbounded domains is the subject of the classical papers of Gilbarg \cite{G}, Hopf \cite{H}, and Serrin \cite{Se}. We prove sharp extensions of the latter results to fully nonlinear equations in Theorems~\ref{PL} and Proposition \ref{PLNH}, below.

Moreover, we prove a stronger result which, to our knowledge, is new even for general linear equations.  Theorem \ref{PLloc} below, which we prove by a blow-up argument, is more general than both (a) and (b) above, and unites these into a single statement. Roughly, it asserts that a subsolution having algebraic growth at a boundary point (resp. infinity) with a rate less than $\alpha^+$ (resp.~$\alpha^-$) must have its fastest growth along the boundary of the domain.
\medskip

In what follows we  denote $\Omega^*:= \{{|x|^{-2}}x : x \in \Omega\}$, the {\it inversion} of $\Omega$.

\begin{thm}\label{PLloc}
Suppose $\Omega$ satisfies \eqref{zetaiden} and $u$ satisfies the inequality
\begin{equation}\label{inequa1}
F(D^2u, Du, x) \leq 0 \; \mbox{ in }\;  \Omega\qquad (\mbox{resp. }\; \Omega^*);
 \end{equation}
Assume that
\begin{equation}\label{condformploc}
\limsup_{x\to 0, x\in \Omega} |x|^{\alpha^+}{u(x)} \le0, \qquad   (\mbox{resp. }\; \limsup_{x\to \infty, x\in \Omega^*} |x|^{\alpha^-}{u(x)} \le0).
\end{equation}
Then
\begin{equation*}
\limsup_{x\to 0, x\in \Omega}u(x)\le \limsup_{x\to 0, x\in \partial\Omega}u^+(x),\quad   (\mbox{resp. }\;\limsup_{x\to \infty, x\in {\Omega}^*}u(x)\le \limsup_{x\to \infty, x\in \partial\Omega^*}u^+(x))
 \end{equation*}
\end{thm}
Similarly to the theorems in the previous subsection, Theorem \ref{PLloc} is valid and will be proved for subsolutions of more general equations, which ``blow up" to  \eqref{inequa1} at the origin (resp. at infinity), in the sense of Section~\ref{ibsec1}.

\medskip

Theorem \ref{PLloc} implies the following extended \pl maximum principle. Here we set $\mathcal{D}= \Omega\cup\widetilde\Omega^*$ where $\Omega$ and $\widetilde\Omega$ are bounded domains satisfying $0\in \partial \Omega\cap\partial \widetilde\Omega$ and such that \eqref{zetaiden} holds for both $\Omega$ and $\widetilde\Omega$, with possibly different $\omega$, $\widetilde\omega$. We denote $\alpha^+:= \alpha^+(F,\Omega)$, $\alpha^-:= \alpha^-(F,\widetilde\Omega)$.

\begin{thm}\label{PL}
Suppose that $\mathcal{D}^\prime\subseteq \mathcal{D}$ is a domain and $u$ is such that
\begin{equation*}
F(D^2u, Du, x) \leq 0 \; \mbox{ in }\;  \mathcal{D}^\prime, \qquad
 u \leq 0  \; \mbox{ on }\;  \partial \mathcal{D}^\prime \setminus \{ 0 \}, \;\mbox{ and}\\
 \end{equation*}
\begin{equation}\label{condformp}
\left\{
\begin{array}{lcl}
\lim_{x\to0} |x|^{\alpha^+}{u(x)} =0, & \mbox{if}&  0\in \partial \mathcal{D}^\prime\\ \lim_{|x|\to \infty} |x|^{\alpha^-}{u(x)}  =0& \mbox{if } &\mathcal{D}^\prime\mbox{ is unbounded}.
\end{array}
\right.
\end{equation}
Then $u \leq 0$ in $\mathcal{D}^\prime$.
\end{thm}

In addition to its sharpness with respect of the growth of $u$, \eqref{condformp} permits $u$ to have singularities a priori at both the origin and infinity. This is a stronger statement in comparison to Phragm\`en-Lindel\"of principles in the literature, and requires a different proof. In fact, we do not actually know of a previous result that is comparable, even for linear equations.

\medskip

As we will see below (Proposition \ref{PLNH}), even stronger results can be obtained if we have a domain with only one singularity, together with information on the solution on the whole boundary of the domain. This is the situation considered in  previous works on maximum principles of \pl type.
 We remark that in the recent years there have been a number of papers on such maximum principles for  fully nonlinear equations. In particular, it was shown by Miller \cite{Mi} (for classical solutions and symmetric cones) and by Capuzzo-Dolcetta and Vitolo \cite{CDV} (for viscosity solutions and more general domains) that there exists a constant $\alpha=\alpha(n,\mu,\Omega)>0$ such that a maximum principle holds at infinity for solutions of $\pucci(D^2u) - (\mu/|x|)|Du|\le 0$ in a conical-type domain $\Omega$, provided $u$ is  $O(|x|^\alpha)$ at infinity (see also \cite{Vi,ARV,Pu}). 

\medskip

Finally, we give another corollary of Theorem~\ref{PLloc} which is an extension of the second theorem of Phragm\`en and Lindel\"of we quoted above. The only previous result of this type of which we are aware is the work by Friedman \cite{Fr}, who studied a class of linear equations in some cones.

\begin{thm}\label{PL2}
 Suppose that $\mathcal{D}^\prime\subseteq \mathcal{D}$ is a domain and $u$ satisfies the inequalities
\begin{equation*}
F(D^2u, Du, x) \leq 0\leq  \widetilde{F}(D^2u, Du, x)\quad \mbox{ in }\;   \mathcal{D}^\prime,
\end{equation*} where $\widetilde{F}$ is the dual operator of $F$, defined in \eqref{Ftilde}, and \eqref{condformp} holds.
Then
$$
\lim_{x\to0, x\in \partial \Omega} {u(x)} =a \quad \mbox{implies}\quad \lim_{x\to 0, x\in \Omega} {u(x)} =a \quad (\mbox{if}\; 0\in \partial \mathcal{D}^\prime),
$$
and
$$\lim_{|x|\to\infty, x\in \partial \Omega^*} {u(x)} =a \quad \mbox{implies}  \quad \lim_{x\to \infty, x\in \Omega^*} {u(x)} =a \quad (\mbox{if} \; \mathcal{D}^\prime \ \mbox{is unbounded}).
$$
\end{thm}

\subsection{Organization of the article}
In the next section, we state the notation and some preliminary results we need. In Section~\ref{alphadefs} we define $\alpha^+$ and $\alpha^-$, and we construct $\Psi^+$ and $\Psi^-$ in Section~\ref{singsing}. Theorem~\ref{uniq} and a special case of Theorem~\ref{PLloc} are proved in Section~\ref{maxpri}. A key ingredient in the proof of Theorem~\ref{ibs0}, a continuous dependence estimate, is proved in Section~\ref{SCDE}, and Theorems~\ref{ibs0}-\ref{poisson} are proved in Section~\ref{ibsec}. The \pl principles are proved in Section~\ref{sectpoisson}.

\section{Preliminaries} \label{prelim}

\subsection{Notation and hypotheses}
Throughout this article, the symbols $C$ and $c$ always denote positive constants which may vary in each occurrence but depend only on the appropriate quantities in the context in which they appear.  If $K \subseteq \R^n$, we write $\Omega \subset\subset K$ if $\Omega$ is a domain and its closure $\bar\Omega$ is a subset of $K$. We denote the Euclidean length of a vector $x\in \R^n$ by $|x|$, and the unit sphere is written $\sph := \{ x \in \R^n : |x|=1 \}$. The open ball centered at $x$ with radius $r>0$ is denoted by $B_r(x)$, and we set $B_r : = B_r(0)$. The sets $\USC(\Omega)$, $\LSC(\Omega)$ and $C(\Omega)$ are, respectively, the set of upper semicontinuous, lower semicontinuous, and continuous functions in $\Omega$. The set of $n$-by-$n$ symmetric matrices is written $\Sy$, and $\iden\in \Sy$ is the identity matrix. We write $M\geq N$ if the matrix $M-N$ is nonnegative definite. For $a,b\in \R^n$, the symmetric tensor product $a\otimes b$ if the matrix with entries $\frac{1}{2}(a_ib_j + b_ia_j)$. If $X$ is a square matrix, we denote the transpose of $X$ by $X^t$.

Given $0 < \lambda \leq \Lambda$, the Pucci extremal operators $\Pucci$ and $\pucci$ are the nonlinear functions $\Sy \to \R$ defined by
\begin{equation*}
\PucciSub(M) := \sup_{A\in \llbracket\elp,\Elp\rrbracket} \left[ - \trace(AM) \right] \quad \mbox{and} \quad \pucciSub (M) := \inf_{A\in \llbracket\elp,\Elp\rrbracket} \left[ - \trace(AM) \right],
\end{equation*}
where $\llbracket\lambda,\Lambda\rrbracket$ is the subset of $\Sy$ consisting of matrices $A$ for which $\lambda \iden \leq A \leq \Lambda \iden$. When performing calculations it is useful to keep in mind that $\Pucci$ and $\pucci$ may also be expressed as
\begin{equation} \label{puccform}
\Pucci(M) = -\lambda \sum_{\mu_j > 0} \mu_j - \Lambda \sum_{\mu_j < 0} \mu_j \quad \mbox{and} \quad \pucci(M) = -\Lambda \sum_{\mu_j > 0} \mu_j - \lambda \sum_{\mu_j < 0} \mu_j,
\end{equation}
where $\mu_1, \ldots, \mu_n$ are the eigenvalues of $M$. These operators satisfy the inequalities
\begin{multline} \label{puccineq}
\pucci(M) + \pucci(N) \leq \pucci(M+N) \leq \pucci(M) + \Pucci(N) \\
\leq \Pucci(M+N) \leq \Pucci(M) + \Pucci(N).
\end{multline}
We also observe that $\Pucci$ and $\pucci$ are rotationally invariant, so that if $U$ is an orthogonal matrix, then $\mathcal{P}^\pm_{\lambda,\Lambda} (U^t M U ) = \mathcal{P}^\pm_{\lambda,\Lambda} (M)$.

For $\pi \subseteq \sph$, we denote by $\co\pi $ the cone-like domain $\co\pi : = \left\{ t y : y \in \pi, \ t > 0   \right\}$. Annular sections of $\co\pi$ are written
\begin{equation*}
E(\pi,r,R) : = \co\pi \cap \left( B_R \!\setminus\! B_r\right), \quad 0 \leq r < R.
\end{equation*}
Throughout, we reserve the symbol $\omega$ to denote a $C^2$-smooth subdomain of $\sph$. We denote spaces of continuous, homogeneous functions on $\co\omega$ by
\begin{equation*}
\hc{\alpha} : = \left\{ u \in C(\co\omega) : u(x) = t^{\alpha} u(tx) \ \mbox{for all}\ x\in \co\omega, \ t>0 \right\},
\end{equation*}
for each number $\alpha \in \R\! \setminus\! \{ 0 \}$. Observe that each function $u\in \hc{\alpha}$ is $-\alpha$-homogeneous and thus determined by its values on $\omega$.

\medskip

Let us state our primary assumptions on the nonlinear operator $F=F(M,p,x)$, which we take to be a continuous function
\begin{equation*}
F:\Sy\times \R^n \times  ( \cob\omega  ) \to \R
\end{equation*}
which is uniformly elliptic, in the sense that for some constants $0 < \elp \leq \Elp$ and $\mu\geq 0$, and every $M,N\in \Sy$, $p,q\in \R^n$, and $x\in \cob\omega$ we have
\begin{multline}\label{ellip}
\pucci(M-N) - \mu |x|^{-1} |p-q| \leq F(M,p,x) - F(N,q,x) \\ \leq \Pucci(M-N) + \mu |x|^{-1} |p-q|.\end{multline}
We assume $F$ is positively 1-homogeneous in $(M,p)$:
\begin{equation} \label{homogen}
F(tM,tp,x)  =  tF(M,p,x) \quad \mbox{for} \ t\geq 0, \ (M,p,x) \in \Sy\times \R^n \times  ( \cob\omega  ).
\end{equation}
In order to apply comparison results (found for example in \cite{JS}) we require $F$ to possess some regularity in $x$; specifically, for some constants $K >0$ and $\theta \in ( \tfrac12,1]$,
\begin{equation} \label{regxh}
\left| F(M,0,x) - F(M,0,y) \right| \leq K\! \left( 1 + |M| \right) |x-y|^\theta \quad \mbox{for all} \ M \in \Sy, \ x,y \in \omega.
\end{equation}
Finally, in order to prove that $F$ possesses homogeneous singular solutions in cones we need to  assume that $F$ is invariant under dilations:
\begin{equation} \label{scaling}
F(r^2M,rp,x) = r^2F(M,p,rx) \quad \mbox{for} \ r >  0, \ (M,p,x) \in \Sy\times \R^n \times  ( \cob\omega  ).
\end{equation}

Note that if $F = F(M,x)$ is independent of $p$ and \eqref{homogen} holds, then  \eqref{scaling} is equivalent to $F(M,x) = F(M,x/|x|)$. Obviously if $F=F(M)$, then  \eqref{ellip}-\eqref{scaling} reduce to uniform ellipticity and positive 1-homogeneity.
Examples of operators which satisfy \eqref{ellip}-\eqref{scaling} are the extremal operators
\begin{equation}\label{extr}
F^\pm(D^2u, Du, x) := \mathcal{P}^\pm_{\elp,\Elp}(D^2u) \pm\mu \frac{1}{|x|} |Du|.
\end{equation}

The scaling relation~\eqref{scaling} ensures that $u$ is a solution of
$
F(D^2u,Du,x) = f(x)$ in $\Omega \subseteq \co\omega$
if and only if the function $u_r(x): = u(rx)$ is a solution of the equation
\begin{equation*}
F(D^2u_r,Du_r,x) = r^2f(rx) \quad \mbox{in} \  \Omega_r:=\left\{ x: rx \in \Omega \right\}
\end{equation*}

 As we will see below, our proofs are made much shorter by the fact that the class of operators satisfying \eqref{ellip}-\eqref{scaling} is invariant with respect to inversions in $\rn\!\setminus\!\{0\}$. We mention in passing that it will be quite clear from our arguments that our results from \cite{ASS} (where we considered only the case $F=F(M)$) extend to operators which satisfy \eqref{ellip}-\eqref{scaling}.

\subsection{The inverted operator} \label{invop}

Given an operator $F=F(M,p,x)$, we define the \emph{inversion} $F^*$ of $F$ by
\begin{multline} \label{Fstar}
F^*(M,p,y) : = \\ F\!\left(\, J(y) M J(y) - 2|y|^{-2} \left( (y\cdot p) J(y) + y\otimes J(y)p + y\otimes p \right),\, J(y) p,\, y \right),
\end{multline}
where the matrix $J = J(y) : = \iden - 2 |y|^{-2} y\otimes y$ is symmetric and orthogonal. The operator $F^*$ has the property that if $u$ is a (sub/super)solution of the equation
\begin{equation*}
F(D^2u(x), Du(x), x) = f(x) \quad \mbox{in} \ \Omega \subseteq \co\omega,
\end{equation*}
then the inverted function
\begin{equation} \label{ustar}
u^*(y) = u(x),\qquad x={|y|^{-2}}{y},
\end{equation}
is a (sub/super)solution of the equation
\begin{equation*}
F^*(D^2u^*(y), Du^*(y) , y) = |y|^{-4} f^*( y) \quad \mbox{in} \ \Omega^* : = \left\{ y\in \co\omega : |y|^{-2} y \in \Omega \right\}.
\end{equation*}
This is easy to check for smooth $u$, and so it holds in the viscosity sense as well since we can perform nearly the same calculation with smooth test functions.

Obviously, $\co\omega^*=\co\omega$, and  $u\in \hc{\alpha}$ if and only if $u^* \in \hc{-\alpha}$. The usefulness of the inverted operator $F^*$ comes from the fact that if  $F$ satisfies \eqref{ellip}-\eqref{scaling}, then $F^*$ does as well (possibly with a larger value of $\mu$ in \eqref{ellip}). For example, we have
\begin{align*}
\lefteqn{F^*(M,p,y) - F^*(N,q,y)} \quad & \\
& \leq \Pucci\left( J(M-N)J - 2|y|^{-2} (2y\otimes (p-q) + y\cdot (p-q) \iden \right.\\
& \qquad \left. + 8|y|^{-4} (y\cdot (p-q))y\otimes y \right) + \mu|y|^{-1}  |Jp-Jq| \\
& \leq \Pucci(M-N) + 2|y|^{-2} \, \Pucci\left( 4|y|^{-2} y\cdot (p-q) y\otimes y - 2y\otimes (p-q)  \right.\\
& \qquad \left. - y\cdot (p-q) \iden \right) +  \mu|y|^{-1}  |p-q| \\
& \leq \Pucci(M-N) + |y|^{-1} \left( 2((n-1)\Lambda - \lambda)+\mu \right) |p-q|.
\end{align*}
The other side of the ellipticity condition is verified in a similar fashion. It is routine to check the scaling and homogeneity relations. The condition \eqref{regxh} is immediate. Finally, a long but routine calculation confirms that we have the duality property
\begin{equation*}
F^{**} = F.
\end{equation*}
In general, $F^*$ has gradient dependence even if $F$ does not.

\medskip

Finally, we remark that $(-\Delta)^*u = -\Delta u +2(2-n) |y|^{-2} y\cdot Du$, and that the function $u^*$ in \eqref{ustar} is the Kelvin transform of $u$ only if the dimension $n=2$. The properties of the inverted operator $F^*$ we need are not related to the Kelvin transform, which is to be expected since this transform does not possess special properties vis-\`a-vis a general nonlinear operator $F$.

\subsection{Several known results}
In this section we state some results for viscosity solutions of uniformly elliptic equations which are well-known or essentially known. We refer to \cite{UG,CC} for an introduction to the theory of viscosity solutions, including basic definitions.

We denote, for constants $0< \lambda \leq \Lambda$ and $\mu,\nu\geq 0$, the extremal operators
\begin{equation*}
\mathcal{L}^+[u]:=\Pucci(D^2 u)+\mu|Du|+\nu |u|\quad \mbox{and} \quad \mathcal{L}^-[u]:=\pucci(D^2 u)-\mu|Du|-\nu |u|.
\end{equation*}

A crucial ingredient in several of our proofs is the following global Harnack inequality for quotients of positive solutions. For strong solutions of general linear equations in nondivergence form this result goes back to the work of Bauman \cite{Ba}. In Appendix~\ref{BHQproof} we give an outline of the proof for viscosity solutions of fully nonlinear equations, since we could not find a reference in the literature.

\begin{prop}[Global Harnack inequality]\label{bhq}
Assume  $\Omega$ is a bounded domain and $\Sigma$ is an open, $C^{2}$ subset of the boundary $\partial \Omega$. Suppose that $u,v \in C(\Omega\cup\Sigma)$ are both positive solutions of the inequalities \eqref{harcon1} such that $u= 0 = v$ on $\Sigma$. Then for each $\Omega' \subset \subset \Omega \cup \Sigma$, we have the estimate
\begin{equation}\label{harres1}
\sup_{\Omega'} \frac{u}{v} \leq C \inf_{\Omega'} \frac{u}{v}.
\end{equation}
The constant $C>1$ depends only on $n$, $\lambda$, $\Lambda$, $\mu$, $\nu$, the curvature of $\Sigma$, a lower bound for $\dist(\Omega',\partial \Omega\!\setminus \! \Sigma)$, and the diameter of $\Omega$ measured in the path distance.
\end{prop}

We will also use the following global H\"older gradient estimate (see Winter~\cite[Theorem 3.1]{W} and Swiech \cite{Swi}). We state it for an operator $G:\Sy\times\R^n\times \R \times\bar\Omega \to \R$ 
 which is continuous and  satisfies the structural conditions
\begin{multline}\label{unifellip}
\pucci(M-N) - \mu|p-q| - \nu|z-w| \leq G(M,p,z,x) - G(N,q,w,x) \\ \leq \Pucci(M-N) + \mu|p-q| + \nu|z-w|,
\end{multline}
and, for some $\mu,\nu,\beta_0\ge 0$,
\begin{equation} \label{regx}
|G(M,0,0,x) - G(M,0,0,y)| \leq \beta_0 |M|  \quad \mbox{for every} \ x,y\in \bar\Omega.
\end{equation}

\begin{prop}[{Global gradient H\"older estimate}] \label{ggh}
Suppose $G$ is continuous and satisfies \eqref{unifellip} and \eqref{regx} above. Suppose that $\Omega$ is a bounded domain and $\Sigma$ is an open, $C^2$ portion of the boundary $\partial \Omega$. Suppose that $u$ satisfies
\begin{equation*}
\left\{ \begin{aligned}
& G(D^2u,Du,u,x) = f &  \mbox{in} & \ \Omega, \\
& u = \varphi & \mbox{on} & \ \Sigma,
\end{aligned} \right.
\end{equation*}
for some $\varphi \in C^{1,\gamma}(\Sigma)$, $f\in L^p(\Omega)$, $p>n$. Then for every $\Omega' \subset\subset \Omega \cup \Sigma$, we have the estimate
\begin{equation*}
\| u \|_{C^{1,\gamma}(\Omega')} \leq C\left( \| u \|_{L^\infty(\Omega)} + \| f \|_{L^\infty(\Omega)}  + \|\varphi\|_{C^{1,\gamma}(\Sigma)}\right)
\end{equation*}
where $C$ depends on $n$, $\lambda$, $\Lambda$, $\mu$, $\nu$, $\beta_0$, the diameter of $\Omega$, the curvature of $\Sigma$, and in the case $\Sigma \neq \partial\Omega$, a lower bound for $\dist(\Omega',\partial \Omega\!\setminus \! \Sigma)$.
\end{prop}

Our construction of singular solutions relies on the following result of Rabinowitz~\cite{R2}, which is a generalization of the Leray-Schauder alternative. A nice proof can also be found for example in \cite{C}.

\begin{prop}\label{rabinowitz}
Let $X$ be a real Banach space, $K\subseteq X$ a convex cone, and $A:[0,\infty) \times K \to K$ a compact and continuous map such that $A(0,u) = 0$ for every $u\in K$. Then there exists an unbounded connected set $S \subseteq [0,\infty)\times K$ with $(0,0)\in S$, such that $A(\alpha,u) = u$ for every $(\alpha,u) \in S$.
\end{prop}

\subsection{Some maximum principles}

We will frequently use the following lemma, which  is a technical tool  enabling us to handle some issues involving the conical boundaries. It asserts that a supersolution, which is positive on a compact subset of an annular slice  $E\subset\mathcal{C}_\omega$, vanishes on the sides of the cone and is not too small near the top and bottom of $E$, must be nonnegative on a smaller annular slice; see Figure~\ref{snapfig}.

\begin{lem} \label{snap}
Assume $\omega' \subset\subset \omega$, $a \geq 0$, $\ep \in \R$, and $u \in \LSC(\bar E(\omega,\tfrac12,4))$ satisfy
\begin{equation*}
\left\{ \begin{aligned}
&\mathcal{L}^+[u] \geq 0 & \mbox{in} & \ E(\omega, \tfrac{1}{2} ,4) \!\setminus\! E(\omega',1,2),  \\
& u \geq a & \mbox{in} & \ E(\omega',1,2),\\
& u \geq 0 & \mbox{on} & \ \partial \co{\omega} \cap (B_4\!\setminus \!B_{1/2}),\\
& u \geq -\ep & \mbox{on} & \ \co\omega \cap \partial \left( B_4\!\setminus\! B_{1/2}\right)
\end{aligned} \right.
\end{equation*}
(see Figure~\ref{snapfig}). Then there exists $\ep_0 = \ep_0(n,\omega,\omega',\elp,\Elp,\mu,\nu)> 0$ such that
\begin{equation*}
\ep \leq \ep_0 a \quad \mbox{implies that} \quad  u \geq 0 \quad \mbox{in} \ E(\omega, 1 ,2).
\end{equation*}
\end{lem}
\begin{proof}
We denote $\Omega : = E(\omega,\textstyle\frac{1}{2},4)  \setminus E(\omega', 1,2 )$.
Let $v_+$ and $v_-$ be the solutions of the Dirichlet problems
\begin{equation*}
\left\{ \begin{aligned}
& \mathcal L^\pm(D^2 v_\pm) = 0 & \mbox{in} & \ \Omega,\\
& v_\pm = g_\pm & \mbox{on} & \ \partial \Omega,
\end{aligned} \right.
\end{equation*}
where $g_+ = g_- = 0$ on the lateral sides $\partial \mathcal C_\omega \cap (B_4\setminus B_{1/2})$ of the outer part of the boundary of $\Omega$, $g_+= 1$ and $g_- = 0$ on the inner boundary $\partial E(\omega',1,2)$, and finally $g_+=0$ and $g_- = 1$ on the top and bottom parts $\mathcal C_\omega \cap \partial\! \left( B_4 \cap B_{1/2} \right)$ of the outer boundary. Elliptic estimates and Hopf's lemma imply that, if $\ep_0 = \ep_0(n,\omega,\omega',\elp,\Elp,\mu,\nu)> 0$  is sufficiently small,  then
\begin{equation*}
v_+ > \ep_0 v_- \quad \mbox{in} \ \Omega \cap \left( B_2 \setminus B_1\right) = \left( \mathcal{C}_\omega \setminus \mathcal C_{\omega'} \right) \cap \left( B_2 \setminus B_1 \right).
\end{equation*}
Set $v: = a v_+ - \ep v_-$, and observe that $u \geq v$ on $\partial \Omega$. Therefore, by the comparison principle, $u \geq v$ in $\Omega$. In particular, $u > 0$ in $\left( \mathcal{C}_\omega \setminus \mathcal C_{\omega'} \right) \cap \left( B_2 \setminus B_1 \right)$ if $\ep<\ep_0a$. Since we have $u \geq a > 0$ in $E(\omega',1,2) = \mathcal C_{\omega'} \cap \left( B_2 \setminus B_1 \right)$, we obtain $u\geq 0$ in $E(\omega,1)$.
\end{proof}
\begin{figure}
\begin{tikzpicture}[>=stealth]
\pgftransformscale{1}
\fill[lightgray] (.4,-.3) arc (-36.8699:36.8699:0.5) -- (3.2,2.4) arc (36.8699:-36.8699:4) -- cycle;
\draw[very thick] (0,0) -- (4,3);
\draw[very thick] (0,0) -- (4,-3);
\draw[thick] (.4,-.3) node[anchor=north]{{$\!\!\!\!\!\!\tfrac12$}} arc (-36.8699:36.8699:0.5);
\draw[thick] (0.96592583,-0.25881905) -- (1.93185165,-0.51763809);
\draw[thick] (0.96592583,0.25881905) -- (1.93185165,0.51763809);
\fill[white] (0.96592583,-0.25881905) -- (1.93185165,-0.51763809) arc (-15:15:2) -- (0.96592583,0.25881905) arc (15:-15:1) -- cycle;
\draw (.8,-.6) node[anchor=north]{{$\!\!\!\!\!1$}} arc (-36.8699:36.8699:1);
\draw (1.6,-1.2) node[anchor=north]{{$\!\!\!\!\!2$}} arc (-36.8699:36.8699:2);
\draw[thick] (3.2,-2.4) node[anchor=north]{$\!\!\!\!\!4$} arc (-36.8699:36.8699:4);
\draw[thick,->] (1,1.6)node[anchor=south]{$E(\omega',1,2)$} -- (1.35, 0);
\draw[thick,->] (4.5,2)node[anchor=west]{$E(\omega,\tfrac12,4)$} -- (3.4, 1);
\draw (4.15, -2.7) node[anchor=south]{$\mathcal{C}_\omega$};
\end{tikzpicture}
\caption{The function $u$ in Lemma~\ref{snap} is a supersolution in the grey region $\Omega$ and positive in the white region illustrated above.}
\label{snapfig}
\end{figure}
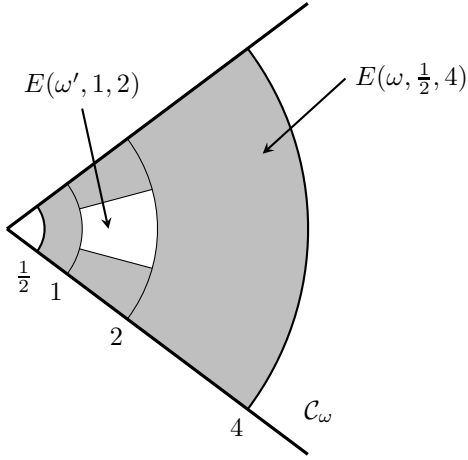

We now give a comparison principle for spaces of homogeneous functions in $\mathcal{C}_\omega$, which is analogous to some results in the principal eigenvalue theory asserting the simplicity of the principal eigenvalue. To prove a result like this, it is typical to use a small domain maximum principle to deal with the behavior close to the boundary.  For this purpose, we use instead Lemma~\ref{snap}.

\begin{prop} \label{hcp}
Suppose that $\alpha > 0$, and $f\in \hc{\alpha+2}$ is nonnegative. Suppose that $u,v\in \hc{\alpha}$ satisfy the differential inequalities
\begin{equation}
F(D^2u,Du,x) \leq f \leq F(D^2v,Dv,x)  \quad \mbox{in} \ \co{\omega},
\end{equation}
as well as
\begin{equation*}
u\leq 0 \ \mbox{on} \ \partial \co{\omega} \setminus \{ 0 \} \qquad \mbox{and} \qquad v > 0 \ \mbox{in} \ \co{\omega}.
\end{equation*}
Then either $u \leq v$ in $\co{\omega}$ or else there exists $t> 1$ such that $u \equiv tv$.
\end{prop}

\begin{proof}
We claim that $u \leq kv$ in $\co{\omega}$ for sufficiently large $k>0$. Choose $\omega'\subset\subset \omega$, and put $a: = \frac{1}{2}\inf_{E(\omega',1,2)} v > 0$. Let $C_0$ be such that $u\ge -C_0$ in $E(\omega,\textstyle\frac12,4)$, and  let $k> 1$ be so large that $u \leq ak$ in $E(\omega',\frac{1}{2},4)$. Define $z: = kv - u$, and observe that
\begin{equation*}
\Pucci(D^2z) + 2\mu|Dz| \geq (k-1)f\geq 0 \quad \mbox{in} \ E(\omega,\textstyle\frac12,4),
\end{equation*}
$z\geq 0$ on $\partial \co\omega \cap (B_4\setminus B_{1/2})$, $z\geq ak$ on $E(\omega',1,2)$, and $z\geq -C_0$ in $E(\omega,\frac12,4)$. Hence $z\geq 0$ in $E(\omega,1,2)$ by Lemma~\ref{snap}, provided we choose $k>C_0(\ep_0a)^{-1}$, where $\ep_0> 0$ is as in Lemma~\ref{snap}. By homogeneity, $z\geq 0$ in $C_\omega$, and the claim is proved.

We have established that the quantity
\begin{equation*}
t : = \inf\left\{ s > 1 : u \leq sv \ \mbox{in} \ \co\omega \right\}
\end{equation*}
is finite. If $t=1$, then we have the first alternative in the conclusion of the proposition, and hence nothing more to show. Suppose instead that $t>1$, and set $w: = tv - u \geq 0$. We must show that $w\equiv 0$. As before, $w$ satisfies
\begin{equation*}
\Pucci(D^2w) + 2\mu|Dw|  \geq 0 \quad \mbox{in} \ E(\omega,\textstyle\frac12,4).
\end{equation*}
If $w\not\equiv 0$, then the strong maximum principle implies $w > 0$ in $\co{\omega}$. In this case, we repeat the argument from the first paragraph. Select $\omega' \subset\subset \omega$ and redefine $a: = \frac{1}{2}\inf_{E(\omega',1,2)} w$. For sufficiently small $0 < \delta < t-1$, the function $\widetilde w:= w - \delta v$ satisfies $\widetilde w \geq 0$ on $\co{\omega'}$ and $\widetilde w \geq a$ on $E(\omega',1,2)$. Choosing $\delta > 0$ smaller still, we may assume that $\widetilde w \geq -\ep_0a$ in $E(\omega,\frac{1}{2},4)$, so  Lemma~\ref{snap} applies and we conclude that $\widetilde w \geq 0$ in $E(\omega,1,2)$. By homogeneity, this implies that $\widetilde w\geq 0$ in $\co{\omega}$. That is, $u \leq (t-\delta) v$, a contradiction to the definition of $t$.
\end{proof}

\section{Definition of $\alpha^+(F,\omega)$ and  $\alpha^-(F,\omega)$} \label{alphadefs}

Define the quantities
\begin{multline} \label{ap}
\alpha^+(F,\omega) : = \sup\left\{ \alpha > 0 :  \ \mbox{there exists} \ u \in \hc{\alpha} \right. \\
\left. \mbox{such that} \ F(D^2u,Du,x) \geq 0 \ \mbox{and} \ u>0 \ \mbox{in} \ \co{\omega} \right\}
\end{multline}
and
\begin{multline} \label{am}
\alpha^-(F,\omega) : = \inf\left\{ \alpha < 0 :  \ \mbox{there exists} \ u \in \hc{\alpha} \right. \\
\left. \mbox{such that} \ F(D^2u,Du,x) \geq 0 \ \mbox{and} \ u>0 \ \mbox{in} \ \co{\omega} \right\}.
\end{multline}
Note these definitions bear some resemblance to the definition of the scaling exponents given in (2.4) of \cite{ASS}, but there are important differences.

Immediate from the definition above is the monotonicity property
\begin{equation} \label{alphmon}
\alpha^+(F,\omega) \leq \alpha^+(F,\omega') \quad \mbox{provided} \ \emptyset \neq \omega' \subseteq\omega \subset \sph.
\end{equation}
It follows from the properties of $F^*$ and the fact that $u\in \hc{\alpha}$ if and only if $u^* \in \hc{-\alpha}$, that we have
\begin{equation} \label{dualalph}
\alpha^-(F,\omega) = -\alpha^+(F^*,\omega).
\end{equation}
Now the reader can see why the inverted operator $F^*$ is useful: it reduces the length of our arguments by half. Thus while we state most of our results for $\alpha^+(F,\omega)$, we easily obtain analogues for $\alpha^-(F,\omega)$ by applying them to the dual operator $F^*$ and using~\eqref{dualalph}.

We must show that the admissible set in \eqref{ap} is nonempty and that $\alpha^\pm(F,\omega)$ are finite. In fact, we obtain positive upper and lower bounds for $\alpha^+(F,\omega)$ in terms of the width of $\omega$ and its compliment. These are summarized in the following two lemmas.

\begin{lem}\label{alphlowbnd}
Suppose that $\omega \subseteq \pi: = \left\{ x \in \sph : x\cdot \xi < \sigma \right\}$, for some $\xi \in \sph$ and $ \sigma < 1$. Then there is a constant $c> 0$ depending only on $n$, $\elp$, $\Elp$, $\mu$ and a lower bound for $1-\sigma$, such that
\begin{equation} \label{alphlowbndeq}
\alpha^+(F,\omega) \geq c.
\end{equation}
\end{lem}

\begin{lem} \label{alphuppbnd}
Suppose that $\omega \supseteq \pi: = \left\{ x \in \sph : x\cdot \xi > \sigma \right\}$, for some $\xi \in \sph$ and $ \sigma < 1$. Then there is a constant $C> 0$ depending only on $n$, $\elp$, $\Elp$, $\mu$, and a lower bound for $1-\sigma$, such that
\begin{equation}  \label{alphuppbndeq}
\alpha^+(F,\omega) \leq C.
\end{equation}
\end{lem}

These two lemmas can be deduced from the results of Miller \cite{Mi} mentioned above. Since the proofs in \cite{Mi} rely on some rather involved ODE techniques, we give simpler proofs of Lemmas~\ref{alphlowbnd} and \ref{alphuppbnd} by exhibiting an explicit subsolution and supersolution of $F=0$ in $\mathcal{C}_\omega$, whose form is of interest in itself. Recalling that
$F^-\le F\le F^+$ with $F^\pm$ defined in \eqref{extr}, we check that the function
\begin{equation}\label{super1}
v(x) : = |x|^{-\alpha} \left( \exp(\kappa) - \exp\left( \kappa |x|^{-1} x_n \right) \right),
\end{equation}
is a solution of $F^-(D^2v, Dv, x)\ge 0$ in $\co\pi$ with $\pi : = \left\{ x \in \sph : x_n < \sigma \right\}$ provided $\alpha>0$ is sufficiently small and $\kappa$ is sufficiently large, which in light of \eqref{ap} gives Lemma \ref{alphlowbnd}. Likewise, the function
\begin{equation}\label{sub1}
w(x) : =\frac{1}{2}\left( |x|^{-\alpha-2} x_n^2 - \sigma^2 |x|^{-\alpha}\right)^2
\end{equation}
is a solution of $F^+(D^2w, Dw, x)\le 0$ in $\co\pi$ with $\pi : = \left\{ x \in \sph : x_n > \sigma \right\}$ provided $\alpha>0$ is sufficiently large, which implies Lemma \ref{alphuppbnd} thanks to the definition of $\alpha^+(F,\omega)$ and Proposition \ref{hcp}. The computations are performed in Appendix~\ref{hell}, and from these we obtain explicit constants in \eqref{alphlowbndeq} and \eqref{alphuppbndeq}.

\section{Singular solutions in a conical domain} \label{singsing}

We proceed with the construction of the singular solutions of \eqref{eq} in $\co{\omega}$. We assume throughout this section that $F$ satisfies the conditions \eqref{ellip}-\eqref{scaling}. We often do not display the dependence of $\alpha^+$ on $F$ and $\omega$.

\begin{lem} \label{strictbelow}
Suppose that $0 < \alpha < \alpha^+$. Then there exists $u \in \hc{\alpha}$ satisfying
\begin{equation}\label{equio}
F(D^2u,Du,x) \geq c |x|^{-2} u \quad \mbox{in} \ \co{\omega} \qquad \mbox{and} \qquad u> 0 \quad \mbox{in} \ \co{\omega},
\end{equation}
where $c: = \frac{1}{2}\lambda\alpha(\alpha^+-\alpha)$.
\end{lem}
\begin{proof}
According to the definition of $\alpha^+$, we can find $\frac{1}{2} (\alpha^++\alpha)\leq \beta \leq \alpha^+$ and a function $v \in \hc{\beta}$ such that $v> 0$ in $\co{\omega}$ and
$
F(D^2v,Dv,x) \geq 0$ in $\co{\omega}$.
We now bend $v$ by defining
\begin{equation*}
u(x) : = \left( v(x) \right)^{1/\tau}, \quad \mbox{for} \quad \tau : = \frac{\beta}{\alpha}> 1 + \frac{\alpha^+-\alpha}{2\alpha},
\end{equation*}
Formally we have
\begin{equation*}
Dv = \tau u^{\tau-1} Du, \quad D^2v= \tau  u^{\tau-1} \left( D^2u + (\tau-1) u^{-1} Du\otimes Du\right),
\end{equation*}
and thus in $\co{\omega}$ we have, by \eqref{ellip},
$$
F(D^2u,Du,x) + (\tau-1)u^{-1}\Pucci(Du\otimes Du)\ge (\tau u^{\tau-1})^{-1}F(D^2v, Dv,x)\ge 0
$$
and thus \eqref{puccform} implies
$$
F(D^2u,Du,x) \ge \lambda(\tau-1)u^{-1}|Du|^2.
$$
From the homogeneity of $u$ it follows that $|Du(x)| \geq \alpha|x|^{-1} u(x)$. Using also that $\tau -1 \geq (2\alpha)^{-1} (\alpha^+-\alpha)$, we get \eqref{equio}.
Our calculation above is merely formal, but it can be justified in the viscosity sense by doing a similar calculation with smooth test functions. This is done for instance in the proof of \cite[Lemma 3.3]{ASS}. We thereby obtain the lemma.
\end{proof}

A consequence of Lemma~\ref{strictbelow} and Proposition~\ref{hcp} is the following maximum principle for subsolutions in $\hc\alpha$.

\begin{cor} \label{eigmp}
Suppose that $0 < \alpha < \alpha^+$ and $u\in \hc{\alpha}$ satisfies the inequality
\begin{equation*}
F(D^2u,Du,x) \leq 0 \quad \mbox{in} \ \co\omega,
\end{equation*}
and $u \leq 0$ on $\partial \co\omega \setminus \{ 0 \}$. Then $u\leq 0$ in $\co\omega$.
\end{cor}
\begin{proof}
According to Lemma~\ref{strictbelow} we can find a function $v\in \hc{\alpha}$ such that $v>0$ and
\begin{equation*}
F(D^2v,Dv,x) \geq c|x|^{-2} v \quad \mbox{in} \ \co\omega,
\end{equation*}
for some $c> 0$. We may apply Proposition~\ref{hcp} to $u$ and any positive multiple of $v$ to deduce that either $u\equiv tv$ for some $t>0$ or else $u \leq tv$ for every $t>0$. Obviously the first alternative implies $v\equiv 0$, a contradiction. Thus the second alternative holds, and we can send $t\to 0$, to infer that $u\leq 0$.
\end{proof}

\begin{lem}\label{alphbet}
Assume that $\alpha,\beta > 0$ and $v\in \hc{\alpha}$ is nonnegative. Then there exists a unique function $u \in \hc{\alpha}$ which is a solution of the problem
\begin{equation}\label{dp}
\left\{ \begin{aligned}
 \lefteqn{F\!\left( D^2u - \alpha(\alpha+2) |x|^{-4} (u-v) x\otimes x ,\, Du + \alpha|x|^{-2} (u-v)x, \, x\right)} & \qquad \qquad \qquad\qquad\qquad \qquad\qquad \qquad\qquad \qquad& & \\
& \quad = (1+\mu) |x|^{-2}(\alpha v - \beta u) + \alpha|x|^{-\alpha-2}  & \mbox{in} & \ \co\omega,\\
& u = 0 & \mbox{on} & \ \partial \co\omega\!\setminus\! \{ 0 \}.
\end{aligned} \right.
\end{equation}
Moreover, $u> 0$ in $\co\omega$, and we have the estimate
\begin{equation} \label{dpest}
\| u \|_{L^\infty(\omega)} \leq \frac{C\alpha}{\beta} \left( 1 + (1+\alpha)\| v \|_{L^\infty(\omega)} \right)
\end{equation}
for a constant $C>0$ which depends only on $n$, $\Lambda$ and $\mu$.
\end{lem}
\begin{proof}
Fix $\alpha,\beta> 0$ and $w\in \hc{\alpha}$ with $w\geq 0$. The zero function is a smooth, strict subsolution of \eqref{dp} since by \eqref{ellip}
\begin{align*}
F(\alpha(\alpha+2) |x|^{-4}v x\otimes x , -\alpha|x|^{-2} vx,x) & \leq -\lambda\alpha(\alpha+2) |x|^{-2} v + \mu\alpha|x|^{-2} v \\
& < (1+\mu)\alpha |x|^{-2} v + \alpha |x|^{-\alpha-2}.
\end{align*}
We claim that the function $w:= k|x|^{-\alpha}\in \hc{\alpha}$ is a smooth supersolution of \eqref{dp} provided that we choose $k>0$ large enough. Observe that
\begin{align*}
Dw & = -\alpha |x|^{-2} w x, \\
D^2w &  = \alpha(\alpha+2) |x|^{-4} w x\otimes x - \alpha |x|^{-2} w \iden,
\end{align*}
and insert into \eqref{dp} to find that
\begin{align*}
\lefteqn{F\!\left( D^2w - \alpha(\alpha+2) |x|^{-4} (w-v) x\otimes x ,\, Dw + \alpha|x|^{-2} (w-v)x, \, x\right)} \qquad \qquad \qquad \qquad & \\
& = F\!\left( -\alpha |x|^{-2} w\iden + \alpha(\alpha+2)|x|^{-4} v x\otimes x, \, -\alpha |x|^{-2} vx, \, x \right) \\
& \geq \alpha n \lambda|x|^{-2} w - \alpha(\alpha+2) |x|^{-2} \Lambda v - \alpha \mu |x|^{-2} v.
\end{align*}
Thus $w$ is a supersolution of \eqref{dp} if
\begin{equation*}
 \alpha n \lambda|x|^{-2} w - \alpha(\alpha+2) |x|^{-2} \Lambda v - \alpha \mu |x|^{-2} v \geq (1+\mu) |x|^{-2} (\alpha v - \beta w) + \alpha |x|^{-\alpha-2},
\end{equation*}
which is equivalent to
\begin{equation*}
(\alpha n \lambda + (1+\mu)\beta) w \geq \alpha\left( |x|^{-\alpha} + \left( \Lambda(\alpha+2) + 2\mu + 1\right) v \right).
\end{equation*}
Thus $w$ is a (strict) supersolution of \eqref{dp} provided that we set
\begin{equation} \label{superk}
k := \frac{\alpha\left( 1 + (\Lambda(\alpha+2) + 2\mu+1) \| v\|_{L^\infty(\omega)} \right)}{\alpha n \lambda + \beta}.
\end{equation}
We use the standard Perron method to build a solution of \eqref{dp}. Let us define
\begin{multline*}
u(x) : = \sup\left\{ \widetilde u(x) \, : \, \widetilde u \in \USC(\overline{\mathcal{C}}_\omega \!\setminus\! \{ 0 \} ) \ \mbox{is a subsolution of} \ \eqref{dp} \ \mbox{such that} \right. \\
\left. \widetilde u \leq w \ \mbox{in} \ \co\omega \ \mbox{and}  \ \widetilde u \leq 0\  \mbox{on}\  \partial \co\omega \!\setminus\! \{ 0 \}\right\}.
\end{multline*}
Clearly $u$ is well-defined and $u \geq 0$ in $\co\omega$ since the zero function is a subsolution of~\eqref{dp}. By the invariance of the equation \eqref{dp} (recall \eqref{scaling}) and the fact that $w\in \hc{\alpha}$, we see that any $\widetilde u$ in the admissible class is such that the function  $x\mapsto r^{\alpha} \widetilde u(rx)$ is also in the admissible class, for any $r>0$. Thus by construction we have that $u$ is $(-\alpha)$-homogeneous. Standard  arguments from the theory of viscosity solutions imply that $u\in C(\coo)$, hence $u\in H_\alpha(\omega)$, and that $u$ is a solution of \eqref{dp}. The estimate \eqref{dpest} follows from $u \leq w$ and \eqref{superk}.

It remains to demonstrate the uniqueness of $u$. This follows from the fact that \eqref{dp} is ``strictly proper" in the sense that if $w$ is any supersolution of \eqref{dp}, then $w(x) + \ep|x|^{-\alpha}$ is a \emph{strict} supersolution for any $\ep > 0$, as it is easy to check. Thus if $u_1,u_2\in \hc{\alpha}$ are two solutions, say with $\ep : = \max_{\omega} (u_1-u_2) > 0$, then the function $u_2 + \ep|x|^{-\alpha}$ is a strict supersolution which touches the solution $u_1$ from above at some point, which is impossible.
\end{proof}

\begin{proof}[Proof of Theorem~\ref{fundycones}]
It suffices to prove the statements only for $\alpha^+$ and $\Psi^+$, since we obtain the corresponding assertions for $\alpha^-$ and $\Psi^-$ by considering the dual operator $F^*$ in place of $F$.

Let $K$ denote the convex cone of nonnegative continuous functions on the closure $\bar \omega$ of $\omega$, which is a subset of the Banach space $C(\bar\omega)$. For each $\alpha > 0$ and $w \in K$, let $A(\alpha, w)\in K$ denote the restriction to $\bar\omega$ of the unique solution  of the problem \eqref{dp} with $v = |x|^{-\alpha}w(\frac{x}{|x|})$ and $\beta = \alpha^+$. We also define $A(\alpha,w) = 0$ for every $w\in K$, and every $\alpha\le 0$. It follows from the estimate \eqref{dpest}, the global H\"older estimates for solutions of uniformly elliptic equations, and the stability of viscosity solutions under local uniform convergence, that $A:[0,\infty) \times K \to K$ is a compact, continuous mapping.

According to Theorem~\ref{rabinowitz} there exists an unbounded, connected subset $S \subseteq [0,\infty) \times K$ such that $(0,0)\in S$, and for every $(\alpha,\widetilde u) \in S$ we have $A(\alpha,\widetilde u) = \widetilde u$. That is, for each $(\alpha,\widetilde u)\in S$, the function $ u(x) : = |x|^{-\alpha} \widetilde u(\frac{x}{|x|})$ is a solution of the problem
\begin{equation*}
\left\{ \begin{aligned}
& F(D^2u, D u, x) = (1+\mu)(\alpha-\alpha^+) |x|^{-2} u + \alpha |x|^{-\alpha-2} & \mbox{in} & \ \co\omega, \\
&  u = 0 & \mbox{on} & \ \partial \co\omega \setminus \{ 0 \}.
\end{aligned} \right.
\end{equation*}

We claim that $S\subseteq [0,\alpha^+] \times K$. Suppose on the contrary that $(\alpha,\widetilde u)\in S$ with $\alpha > \alpha^+$. We see immediately that $u\in \hc{\alpha}$ satisfies the inequality
\begin{equation*}
F(D^2u,Du,x) \geq 0 \quad \mbox{in} \ \co\omega.
\end{equation*}
By the strong maximum principle, we have $u>0$. The definition of $\alpha^+$ implies that $\alpha \leq \alpha^+$, a contradiction. Thus we have the claim.

By the unboundedness of $S$, for each $j \geq 1$ there exists $0 < \alpha_j \leq \alpha^+$ and $u_j \in \hc{\alpha_j}$ such that $\| u_j \|_{L^\infty(\omega)} \geq j$ and $A(\alpha_j,\widetilde u_j) = \widetilde u_j$, that is,
\begin{equation*}
\left\{ \begin{aligned}
& F(D^2u_j, Du_j, x) = (1+\mu)(\alpha_j-\alpha^+) |x|^{-2} u_j + \alpha_j |x|^{-\alpha_j-2} & \mbox{in} & \ \co\omega, \\
& u_j = 0 & \mbox{on} & \ \partial \co\omega \! \setminus \! \{ 0 \}.
\end{aligned} \right.
\end{equation*}
By taking a subsequence if necessary, we may assume that $\alpha_j \rightarrow \bar\alpha \in [0,\alpha^+]$. We claim that $\bar\alpha > 0$. Indeed, using $\| u_j \|_{L^\infty(\omega)}\to\infty$ and the estimate \eqref{dpest} with $u=v=u_j$ and $\beta=\alpha^+$, we deduce that $\bar\alpha\ge c\min\{ \alpha^+, \sqrt{\alpha^+} \}$, for some $c>0$ depending only on $n,\Lambda$, and $\mu$.

By defining $v_j (x) : = \| u_j \|^{-1}_{L^\infty(\omega)} u_j(x)$ and taking another subsequence we may suppose that $v_j \rightarrow v$ locally uniformly in $\co\omega$, and uniformly on $\bar\omega$. Observe that $v$ is nonnegative, $\| v\|_{L^\infty(\omega)}=1$, and $v\in \hc{\bar\alpha}$. Dividing the equation for $u_j$ by $\| u_j\|_{L^\infty(\omega)}$ and passing to limits in the viscosity sense, we deduce that $v$ satisfies
\begin{equation} \label{gold}
\left\{ \begin{aligned}
& F(D^2v, Dv, x) = (1+\mu)(\bar \alpha -\alpha^+) |x|^{-2} v & \mbox{in} & \ \co\omega, \\
& v = 0 & \mbox{on} & \ \partial \co\omega\! \setminus\! \{ 0 \}.
\end{aligned} \right.
\end{equation}
By the strong maximum principle, $v> 0$ in $\co\omega$. It remains to show that $\bar\alpha =\alpha^+$. This is easy, since if on the contrary $\bar\alpha < \alpha^+$, then we obtain a contradiction to Corollary~\ref{eigmp}. Setting $\Psi^+ := v$, the proof is complete.
\end{proof}

\section{Proof of Theorem \ref{uniq}}\label{maxpri}

The main results proved in this section are Theorem \ref{uniq} and Theorem \ref{PLloc} in the particular case $\Omega= \co\omega$.
From now on we assume that $\Psi^+$ is normalized so that
\begin{equation}\label{normmax}
\max_{\co\omega\cap\partial B_1}\Psi^+ = 1,\qquad\min_{\co{\omega^\prime}\cap\partial B_1}\Psi^+ \ge 1/2,
\end{equation}
where $\omega^\prime$ is some fixed smooth proper subdomain of $\omega$.

\begin{proof}[Proof of Theorem \ref{PLloc} in the case $\Omega= \co\omega$]
We may assume  $A:=\limsup_{x\to 0}u^+(x)$ is finite, since otherwise we have nothing to prove. We can replace $u$ by $u^+$ since the maximum of subsolutions is a subsolution, so we can also assume $u\ge 0$ in $\Omega$.

We set $M(r)= \sup_{\co\omega\cap \partial B_r}u$ and claim that $r \mapsto M(r)$ is nondecreasing in the interval $(0,1/2)$. Fix $r\in (0,1/2)$ and $\delta > 0$, and for each $0 < s < 1/4$, consider the function
\begin{equation*}
w_s(x) : = s^{\alpha^+}\left(\delta \Psi^+(sx) - u(sx) +A_s\right)= \delta \Psi^+(x) - s^{\alpha^+}u(sx) +s^{\alpha^+}A_s,
\end{equation*}
where $A_s:= \sup_{(\partial \co{\omega}) \cap (B_{4s}\!\setminus \!B_{s/2})} u$, $\lim_{s\to 0}A_s= A$.
By \eqref{ellip} and \eqref{homogen} $w_s$ satisfies
\begin{equation*}
\Pucci(D^2w_s) +2\mu|Dw_s| \geq 0 \quad \mbox{ in }  \; E(\omega,\textstyle\frac12,4).
\end{equation*}
Obviously $w_s\ge0$ on $\partial \co{\omega} \cap (B_4\!\setminus \!B_{1/2})$.
The hypothesis \eqref{condformploc}  clearly implies that for each $\delta_1>0$ we can find $s_1>0$ such that for all $s\in (0,s_1)$
$$
w_s\ge \delta \Psi^+(x) - \delta_1\ge \delta2^{-\alpha^+}-\delta_1 \quad \mbox{ on }  \; \partial E(\omega',1,2),
$$
and of course $w_s\ge -\delta_1$ on $\co\omega \cap \partial \left( B_4\!\setminus \!B_{1/2}\right)$, since $\Psi^+$ is positive.

We can now apply Lemma \ref{snap},  fixing $\delta_1= \varepsilon_02^{-\alpha^+-1}$, where $\ep_0$ is the number from that lemma, to conclude that $w_s\ge 0$  in $E(\omega,1,2)$ for sufficiently small $s>0$. It follows that
\begin{equation*}
u \leq \delta \Psi^+ +s^{\alpha^+}A_s+ M(r) \quad \mbox{on} \ \partial E(\omega, s, r).
\end{equation*}
By the maximum principle, we deduce that
\begin{equation*}
u \leq \delta \Psi^+ +s^{\alpha^+}A_s+ M(r) \quad \mbox{in} \ E(\omega,s,r).
\end{equation*}
This holds for all small enough $s> 0$, and so we may pass to the limit $s\to 0$ and then send $\delta \to 0$ to obtain $u \leq M(r)$ in $\co\omega \cap B_r$. Thus $M(r) = \sup_{\co\omega\cap B_r}u$, so $r\mapsto M(r)$ is nondecreasing.

It follows that $\delta := \inf_{0<r<1/2} M(r) = \lim_{r\searrow 0} M(r)$. To complete the proof, we must show that $\delta = 0$. Suppose on the contrary that $\delta > 0$. Denote $u_r(x) : = u(rx)$ and let $\tilde u_r$ be the solution of the Dirichlet problem
\begin{equation*}
\left\{ \begin{aligned}
& F(D^2\tilde u_r,D\tilde u_r, x) = 0 & \mbox{in} & \ E(\omega,1/2,4), \\
& \tilde u_r = u_r & \mbox{on} & \ \partial E(\omega,1/2,4).
\end{aligned} \right.
\end{equation*}
Obviously, by the maximum principle
\begin{equation}\label{mp1}
\sup_{E(\omega,1/2,4)} u_r\ge \sup_{\partial E(\omega,1/2,4)} u_r = \sup_{E(\omega,1/2,4)} \tilde u_r \ge \tilde u_r \ge u_r\quad\mbox{ in }\; E(\omega,1/2,4),\end{equation} and hence $  \sup_{E(\omega,1/2,4)} \tilde u_r = \sup_{E(\omega,1/2,4)} u_r$.
It follows that for sufficiently small $r>0$, we have $0 \leq \tilde u_r \leq 1+\delta$ in $E(\omega,1/2,4)$. By H\"older regularity, we may assume that $\tilde u_r \to \tilde u$ uniformly on $E(\omega,1/2,4)$ as $r\to 0$, and $\tilde u$ solves the same PDE as $\tilde u_r$. It is clear that $\max_{E(\omega,1/2,4)} \tilde u = \delta$ and there is a point $x_0\in \co\omega\cap \partial B_1$ such that $\tilde u(x_0) = \delta$ (recall $\tilde u = \tilde u_r = 0$ on $\partial \co\omega $). Hence  the strong maximum principle implies that $\tilde u \equiv \delta$. This is impossible since $\tilde u$ vanishes on $\partial \co\omega \cap (B_4\setminus B_{1/2})$.

To prove Theorem \ref{PLloc} in case we have a condition at infinity,
using \eqref{dualalph}, we observe that $u^*(y) =  u(\frac{y}{|y|^2})$ satisfies the condition at zero with $F^*$ in place of $F$, since
\begin{equation*}
\limsup_{|x|\to \infty} |x|^{\alpha^-(F)} u(x) = \limsup_{y\to 0} |y|^{\alpha^*(F^*)}u^*(y).
\end{equation*}
Thus we conclude by what we already proved.
\end{proof}

Next, for each $\pi \subseteq \omega$ and a given function $u$ we set
\begin{equation*}
q^\pm(\pi,r) : = \inf_{E(\pi,r,2r)} \frac{u}{\Psi^\pm}, \qquad\qquad Q^\pm(\pi,r): = \sup_{E(\pi,r,2r)} \frac{u}{\Psi^\pm}.
\end{equation*}
We  write  $Q^\pm(r) = Q^\pm(\omega,r)$, and similarly for  $q^\pm$.

In the  following  two simple lemmas we describe some monotonicity properties of the quantities $q^\pm$ and $Q^\pm$, which play a crucial role in the proof of  Theorem \ref{uniq}.

\begin{lem} \label{PLorange}
Suppose that $u\in\USC(\bar{\mathcal C}_\omega \! \setminus \! \{ 0 \})$ satisfies
\begin{equation*}
\left\{ \begin{aligned}
& F(D^2u,Du,x) \leq 0 & \mbox{in} & \ \co{\omega}, \\
& u \leq 0 & \mbox{on} & \ \partial \co\omega \!\setminus \!\{ 0 \}.
\end{aligned} \right.
\end{equation*}
\begin{enumerate}
\item[(i)] If $\limsup_{|x|\to\infty} u(x) \leq 0$, then $r\mapsto Q^+(r)$ is nonincreasing on $(0,\infty)$;
\item[(ii)] If $\limsup_{x\to 0} u(x) \leq 0$, then $r\mapsto Q^-(r)$ is nondecreasing on $(0,\infty)$.
\end{enumerate}
\end{lem}
\begin{proof}
(i) Observe that for each $\ep > 0$ and $r>0$, there exists $R>r$ such that
\begin{equation*}
u \leq \ep + Q^+(r) \Psi^+ \quad \mbox{on} \ \partial E(\omega,r,R).
\end{equation*}
By the maximum principle, $u\leq \ep + Q^+(r) \Psi^+$ in $E(\omega,r,R)$. Sending $R\to \infty$ and then $\ep \to 0$ we obtain $u \leq Q^+(r) \Psi^+$ in $\co\omega \setminus B_r$, that is, $Q^+(r)=\sup_{\co\omega \setminus B_r} (u/\Psi^+)$. Thus $Q^+(s) \leq Q^+(r)$ for all $0 < r < s$.

(ii) Apply (i) to $u^*$ and $F^*$ and rewrite the result in terms of $u$ and $F$.
\end{proof}

\begin{lem}\label{PLpurple}
Suppose that $u\in\LSC(\bar{\mathcal C}_\omega \! \setminus \! \{ 0 \})$ satisfies
\begin{equation*}
\left\{ \begin{aligned}
& F(D^2u,Du,x) \geq 0 & \mbox{in} & \ \co{\omega}, \\
& u \geq 0 & \mbox{on} & \ \partial \co\omega\! \setminus\! \{ 0 \}.
\end{aligned} \right.
\end{equation*}
\begin{enumerate}
\item[(i)] If $\liminf_{|x|\to\infty} u(x) \geq 0$, then $r\mapsto q^+(r)$ is nondecreasing on $(0,\infty)$;
\item[(ii)] If $\liminf_{x\to 0} u(x) \geq 0$, then $r\mapsto q^-(r)$ is nonincreasing on $(0,\infty)$.
\end{enumerate}
\end{lem}
\begin{proof}
(i) For each $\ep > 0$ and $r> 0$, choose $R > 0$ so large that
\begin{equation*}
u \geq q^+(r) \Psi^+ - \ep \quad \mbox{on} \ \partial E(\omega,r,R).
\end{equation*}
Using the maximum principle and sending $R\to \infty$ and then $\ep \to 0$, we obtain $u \geq q^+(r) \Psi^+$ in $\co\omega\setminus B_r$. Thus $r\mapsto q^+(r)$ is nondecreasing.

(ii) Apply (i) to $u^*$ and $F^*$ and rewrite the result in terms of $u$ and $F$.
\end{proof}

We now proceed to the proof of Theorem~\ref{uniq}, which  relies on the global Harnack inequality for quotients (Proposition~\ref{bhq}) and the previous lemmas.

\begin{proof}[{Proof of Theorem~\ref{uniq}}]
We prove only (i), since (ii) is obtained by a similar argument, or alternatively by applying (i) to $F^*$ and $u^*$.

Recall the function $\hat u_r(x) = u(rx)$ satisfies the same equation as $u$, by the hypothesis \eqref{scaling}. The global Harnack inequality for quotients (Proposition~\ref{bhq}) then implies
\begin{equation} \label{harsnap}
Q^+(r) \leq C q^+(r),
\end{equation}
for some $C> 0$ which does not depend on $r$. The monotonicity of $Q^+$ and $q^+$ given by Theorem \ref{PLloc} and Lemmas~\ref{PLorange} and \ref{PLpurple} then yield
\begin{equation} \label{cqQC}
0 < c \leq q^+(r) \leq Q^+(r) \leq C.
\end{equation}
as well as $q^+(0) \leq q^+(r) \leq Q^+(r) \leq Q^+(0)$, where $q^+(0) : = \lim_{r\searrow 0} q^+(r)$ and $Q^+(0) : = \lim_{r\searrow 0} Q^+(r)$. Therefore to obtain $u\equiv t\Psi^+$, it suffices to show that $t: = Q^+(0) = q^+(0)$. For this we use a rescaling (blow-up) argument. Define the function $u_r(x) : = r^{\alpha^+} u(rx)$. By the monotonicity of $q^+(r)$ and $Q^+(r)$, for every $k>0$ we have
\begin{equation*}
q^+(kr) \Psi^+ \leq u_r \leq Q^+(kr) \Psi^+ \quad \mbox{in} \ \co\omega \setminus B_{k},
\end{equation*}
and
\begin{equation*}
\inf_{E(\omega,k,2k)} \frac{u_r}{\Psi^+} = q^+(kr), \quad \sup_{E(\omega,k,2k)} \frac{u_r}{\Psi^+} = Q^+(kr).
\end{equation*}

Using H\"older estimates and passing to the limit along a subsequence $r \to 0$, we obtain a function $u_0$ which satisfies $F(D^2u_0,Du_0,x) = 0$ in $\co\omega$, $u_0 = 0$ on $\partial \co\omega \!\setminus \! \{ 0 \}$,
\begin{equation*}
q^+(0) \Psi^+ \leq u_0 \leq Q^+(0) \Psi^+\quad \mbox{in} \ \co\omega,
\end{equation*}
and for every $k > 0$,
\begin{equation} \label{bueq}
\inf_{E(\omega,k,2k)} \frac{u_0}{\Psi^+} = q^+(0), \quad \sup_{E(\omega,k,2k)} \frac{u_0}{\Psi^+} = Q^+(0).
\end{equation}
We will show that $u_0 \equiv Q^+(0) \Psi^+$, from which $Q^+(0) = q^+(0)$ follows. Suppose on the contrary that $u_0 \not\equiv Q^+(0) \Psi^+$, so $u_0 < Q^+(0) \Psi^+$ in $\co\omega$. Set $a: = \frac{1}{2}\inf_{E(\omega',1,2)} (Q^+(0)\Psi - u_0)>0$. Then for all $\delta > 0$, the function $$\tilde v(x) : = (Q^+(0) - \delta) \Psi^+(x) - u_0(x)\ge -\delta \Psi^+(x)$$ satisfies
\begin{equation*}
 \Pucci(D^2\tilde v) +2\mu|D\tilde v| \geq 0 \quad   \mbox{ in } \; E(\omega,\textstyle \frac12,4)
 \end{equation*}
  as well as $ \tilde v = 0$ on $\partial \co\omega \cap (B\!\setminus\! B_{1/2})$,  $\tilde v \geq \frac{1}{2} a $ on $E(\omega',1,2)$,
 and $\tilde v \geq -\delta 2^{-\alpha^+}$ on $ E(\omega,\textstyle\frac12,4)$.

Thus if $\delta > 0$ is sufficiently small with respect to $a$, Lemma \ref{snap} applies and we can deduce that $\tilde v\geq 0$ in $E(\omega,1,2)$. This contradicts \eqref{bueq}, and confirms that $Q^+(0) = q^+(0)$. If we let $t$ denote this value, then by the monotonicity properties of $Q^+$ and $q^+$, we deduce that $Q^+\equiv q^+ \equiv t$ on $(0,\infty)$. That is, $u\equiv t \Psi^+$.
\end{proof}

\section{Continuous dependence estimates} \label{SCDE}

In this section, we state and prove a continuous dependence estimate which has an important role in the classification of isolated boundary singularities. To simplify the presentation, we do not attempt to achieve maximal generality in the structural hypotheses below, and we state the result in a form well-suited to its primary end, which is in the proof of Theorem~\ref{ibs0} in Section~\ref{ibsec}. Nevertheless, in view of further applications and since no technical complications arise, we consider operators which may include zero-order or inhomogeneous terms.

We consider an operator $F= G_0$ which satisfies the assumptions \eqref{ellip}--\eqref{scaling} as well as a family of continuous functions
\begin{equation*}
G_r: \Sy\times\R^n \times \R\times (\co\omega\setminus B_1)\to \R
\end{equation*}
which converge to $G_0$ as $r\to 0$ in the sense that there exists $\epsilon\in (0,1)$ such that, for all $M\in\Sy$, $p\in \rn$, $z\in \mathbb{R}$, and $x,y\in\coo$ such that $1\le |x|,|y|\le r^{-\epsilon}$,
\begin{equation} \label{closeness}
\left| G_r(M,p,z,x) - G_0(M,p,y) \right| \leq C|p| |x-y|^\epsilon +  r^\epsilon (1+|z|+|p|+|M|).
\end{equation}

To prove Theorem~\ref{ibs0}, we need a continuous dependence estimate to control the separation between solutions of $G_r=0$ and $G_0=0$, for small $r$. It is well-known how to obtain such estimates using classical viscosity solution methods \cite{UG,IL}. Indeed, such estimates have been obtained for degenerate fully nonlinear elliptic equations under various hypotheses, see for instance \cite{JK}.

However, to prove Theorem~\ref{ibs0}, we need more than usual. In particular, we need to control the \emph{ratios} of the solutions up to the boundary on which they vanish as well as allow for perturbations of the boundary values far away from the region which concerns us (see Figure~\ref{contdepfig}). It is for this reason that the estimate recorded in Proposition~\ref{CDE} has a rather unusual form.

For convenience and with the application to the proof of Theorem~\ref{ibs0} in mind, we state the continuous dependence estimate for perturbations of $\Psi^+$, where $\Psi^+$ is the function given in Theorems \ref{fundycones} and \ref{uniq} for the operator $F=G_0$. Thanks to Proposition \ref{ggh} and the fact that $\Psi^+$ is $-\alpha^+$-homogeneous for some $\alpha^+>0$, we have that $\Psi^+  \in C^{1,\gamma}(\coo\setminus\{0\})$, for some $\gamma > 0$ depending only on $n$, $\lambda$, $\Lambda$ and $\mu$, as well as the estimate
\begin{equation} \label{Psic1alph}
\| \Psi^+\|_{C^{1,\gamma}(E(\omega,1,\infty))} \leq C.
\end{equation}

\medskip

The continuous dependence estimate is the following.

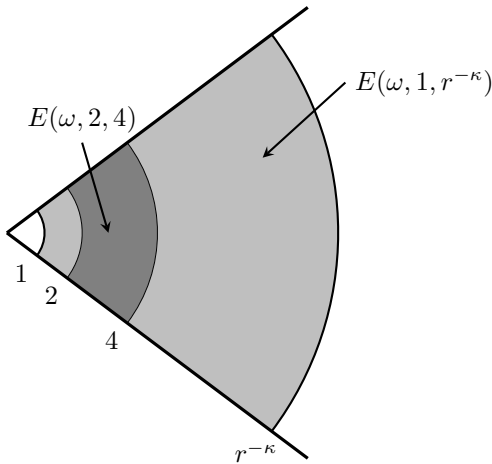
\begin{figure}
\begin{tikzpicture}[>=stealth]
\pgftransformscale{1}
\fill[gray] (.8,-.6) arc (-36.8699:36.8699:1) -- (1.6,1.2) arc (36.8699:-36.8699:2) -- (0.8,-0.6);
\fill[lightgray] (.4,-.3) arc (-36.8699:36.8699:0.5) -- (0.8,0.6) arc (36.8699:-36.8699:1) -- cycle;
\fill[lightgray] (1.6,-1.2) arc (-36.8699:36.8699:2) -- (3.52,2.64) arc (36.8699:-36.8699:4.4) -- cycle;
\draw[very thick] (0,0) -- (4,3);
\draw[very thick] (0,0) -- (4,-3);
\draw[thick] (.4,-.3) node[anchor=north]{{$\!\!\!\!\!\!\!1$}} arc (-36.8699:36.8699:0.5);
\draw[very thin] (.8,-.6) node[anchor=north]{{$\!\!\!\!\!\!\!2$}} arc (-36.8699:36.8699:1);
\draw[very thin] (1.6,-1.2) node[anchor=north]{{$\!\!\!\!\!\!\!4$}} arc (-36.8699:36.8699:2);
\draw[thick] (3.52,-2.64) node[anchor=north]{$\!\!\!\!\!\!\!r^{-\kappa}$} arc (-36.8699:36.8699:4.4);
\draw[thick,->] (1,1.2)node[anchor=south]{$E(\omega,2,4)$} -- (1.35, 0);
\draw[thick,->] (4.5,2)node[anchor=west]{$E(\omega,1,r^{-\kappa})$} -- (3.4, 1);
\end{tikzpicture}
\caption{The domain on which we prove the continuous dependence estimate is illustrated in grey. In the dark grey region, we can control the \emph{ratio} between the two solutions.}
\label{contdepfig}
\end{figure}

\begin{prop} \label{CDE} For each $K_0>0$, there exist positive constants $\kappa, \beta,\bar\beta, r_0 >0$ depending on $n, \lambda, \Lambda, \mu$, $\gamma$, $K_0$, the curvature of $\omega$ and $\epsilon$ (as in \eqref{closeness}), such that if $\psi\in C^1(E(\omega,1,r^{-\kappa}))$ is a solution of
\begin{equation}\label{psirdirp}
\left\{ \begin{aligned}
& G_r(D^2\psi,D\psi,\psi,x) = 0 & \mbox{in} & \ E(\omega,1,r^{-\kappa}), \\
& \psi = \Psi^+ + h & \mbox{on} & \ \partial E(\omega,1,r^{-\kappa}), \\
\end{aligned} \right.
\end{equation}
for some function $h\in C(\partial E(\omega,1,r^{-\kappa}))$ satisfying $|h|\le 1$, and
\begin{equation}\label{psic1alph}
\|\psi\|_{C^{0,1}(E(\omega,1,r^{-\kappa}))}\le K_0,
\end{equation}
then, for all $0 < r \leq r_0$,
\begin{equation}\label{diffnce}
|\psi-\Psi^+ | \leq C r^\beta + \max_{\partial E(\omega,1,r^{-\kappa})}|h|\quad \mbox{in} \ E(\omega, 1, r^{-\kappa}).
\end{equation}
If, in addition, $h=0$ on $\partial \co\omega\cap \partial E(\omega,3/2,6)$ and
\begin{equation}\label{psic1alph2}
\|\psi\|_{C^{1,\gamma}(E(\omega,3/2,6))}\le K_0,
\end{equation}
then, for all $0 < r\leq r_0$,
\begin{equation} \label{ratioestt}
\left|\frac{\psi}{\Psi^+} - 1\right| \leq C \left( r^\beta + \max_{\partial E(\omega, 1, r^{-\kappa})} |h| \right)^{\bar \beta} \quad \mbox{in} \ E(\omega,2,4).
\end{equation}
\end{prop}

\begin{proof}
For convenience we write $R=r^{-\kappa}$ and $V_r : = E(\omega,1,R) = E(\omega,1,r^{-\kappa})$.
Let $\varphi(x)\in C^2(V_r)$  be the solution of the following boundary value problem
\begin{equation} \label{strictifier}
\left\{
\begin{aligned}
& \pucci(D^2\varphi) - \frac{\mu}{|x|}|D\varphi| = 1 & \mbox{in} &  \  V_r,\\
& \varphi = 0 & \mbox{on} & \ \partial V_r.
\end{aligned}
\right.
\end{equation}
The ABP inequality (see Proposition~\ref{abp} below, and the remarks that follow it) yields
 \begin{equation}\label{rrr}
0 \leq \varphi \leq R^4\quad\mbox{ in }\;V_r,\end{equation}
for sufficiently small $r >0$. Define the auxiliary function $\xi: \overline V_r \times \overline V_r \to \R$ by
\begin{equation*}
\xi(x,y) : = \Psi^+(x) - \psi(y) - k \varphi(x) - \frac{A}{2\delta} |x-y|^2  - \max_{\partial E(\omega,1,r^{-\kappa})} |h| - \delta,
\end{equation*}
for some positive constants $\delta,k,A > 0$, to be selected below. Observe that
\begin{equation} \label{xibndry}
\xi(x,x) \leq -\delta \quad \mbox{for each} \quad x\in \partial V_r.
\end{equation}

We will choose the constants $\delta,k,A$ appropriately so that $\xi \leq 0$ on $\overline V_r \times \overline V_r$.
Suppose, on the contrary, that there exists $(x_0,y_0) \in \overline V_r \times \overline V_r$ such that
\begin{equation*}
\xi(x_0,y_0) = \sup_{\overline V_r \times \overline V_r} \xi(x,y) > 0.
\end{equation*}
It is not difficult to show in this case that
\begin{equation} \label{x0y0cls}
|x_0 - y_0 | \leq \bar{C} A^{-1}\delta,
\end{equation}
where $\bar{C}$ depends only on the Lipschitz constants of $\Psi^+$ and $\psi$ (which we are fixed in \eqref{Psic1alph} and \eqref{psic1alph}). Specifically, if one of the points $x_0$ and $y_0$ belongs to $\partial V_r$, then \eqref{x0y0cls} follows from $h=\psi-\Psi^+$ on $\partial V_r$ and
$$
\frac{A}{2\delta} |x_0-y_0|^2<\Psi^+(x_0) - \psi(y_0)    - \max_{\partial E(\omega,1,r^{-\kappa})} |h|.$$
If $x_0,y_0\in V_r$, then \eqref{x0y0cls} follows from $\frac{\partial \xi}{\partial y}(x_0,y_0)=0$.

If one of $x_0$, $y_0$ belongs to $\partial V_r$, then  \eqref{xibndry} and the mean value theorem imply
$$
\delta\le {C}|x_0-y_0|,
$$
where $C$ depends only on \eqref{Psic1alph} and \eqref{psic1alph}. Combining this inequality with \eqref{x0y0cls}, we see that we can fix $A>0$ sufficiently large to get a contradiction. Hence we may assume that $x_0,y_0\in V_r$.

By the maximum principle for semicontinuous functions (see \cite[Theorem 3.2]{UG} and \cite[Proposition II.3]{IL}), there exist matrices $X,Y\in \Sy$ such that
\begin{equation} \label{matmad}
-\frac{3A}{\delta} \begin{pmatrix} \iden & 0 \\ 0 & \iden \end{pmatrix} \leq \begin{pmatrix} X & 0 \\ 0 & -Y \end{pmatrix} \leq \frac{3A}{\delta} \begin{pmatrix} \iden & -\iden \\ -\iden & \iden \end{pmatrix}
\end{equation}
and
\begin{equation} \label{jetmad}
\left( \frac{x_0-y_0}{\delta}, \, X \right) \in \overline{\mathcal{J}}^{2,+} \!\left( \Psi^+\! - k \varphi\right) (x_0), \quad  \left( \frac{x_0-y_0}{\delta} , \, Y \right) \in \overline{\mathcal{J}}^{2,-}\psi(y_0).
\end{equation}
See \cite{UG} for the definition and basic properties of the semijets $\overline{\mathcal{J}}^{2,\pm}$ and their relation to viscosity solutions. The second matrix inequality in \eqref{matmad} and simple linear algebra facts (see \cite[ Proposition II.3 and {Lemma III.1}]{IL}) imply that $X \leq Y$, as well as
\begin{equation} \label{XYord}
|X|, |Y| \leq C \!\left( \delta^{-1/2} \trace(Y-X)^{1/2} + \trace(Y-X) \right), \quad C = C(n,A).
\end{equation}
Since $\varphi\in C^2$, the first inclusion on \eqref{jetmad} may be rewritten as
\begin{equation*}
\left( \frac{x_0-y_0}{\delta}+ k D\varphi(x_0), \, X + k D^2\varphi(x_0) \right) \in \overline{\mathcal{J}}^{2,+} (\Psi^+)(x_0).
\end{equation*}
Since $\Psi^+$ and $\psi$ are, respectively, solutions of $F=0$ and \eqref{psirdirp}, we obtain
\begin{equation}\label{Fineq}
F\!\left( X + k D^2\varphi(x_0),\frac{x_0-y_0}{\delta}+ k D\varphi(x_0),x_0 \right) \leq 0,
\end{equation}
and
\begin{equation} \label{Gineq}
G_r\!\left( Y ,  \, \frac{1}{\delta} (x_0 - y_0) , \psi(y_0), y_0 \right) \geq 0.
\end{equation}
The uniform ellipticity of $F$ and \eqref{strictifier} yield
\begin{align*}
F\left(Y,  \, \frac{1}{\delta} (x_0 - y_0) , x_0\right) & \leq F(X + k D^2\varphi(x_0),\frac{x_0-y_0}{\delta}  + k D\varphi(x_0),x_0)\\ & \qquad + \Pucci(Y-X) - k \Big(\pucci(D^2\varphi(x_0)) -\frac{\mu}{|x|}|D\varphi(x_0)|\Big)\\
& \leq -\lambda \trace(Y-X) -k.
\end{align*}
Set $t: = \trace(Y-X)$. Subtracting the last inequality from \eqref{Gineq} and using \eqref{closeness}, we arrive at
\begin{equation*}
k + \lambda t \leq  C\delta^{-1} |x_0-y_0|^{1+\epsilon} + Cr^\epsilon \Big( 1+K_0 +\frac{|x_0-y_0|}{\delta}+ |Y| \Big).
\end{equation*}
Define $\delta := r^{\epsilon}$. The previous inequality, \eqref{x0y0cls} and \eqref{XYord} yield, for small $r>0$,
\begin{equation} \label{dour}
k + \lambda t  \leq C\delta^\ep + C\delta (1+t)(1+\delta^{-1/2}) \leq C_1r^{\epsilon^2} + C_2r^{\epsilon/2}(1+t).
\end{equation}
If $r > 0$ is sufficiently small, then $ C_2r^{\epsilon/2}<\min\{\lambda, C_1r^{\epsilon^2/2}\}$.  Hence
\begin{equation*}
k \leq 2C_1r^{\epsilon^2/2}.
\end{equation*}
Selecting $k:=4C_1r^{\epsilon^2/2}$ yields a contradiction.

We have shown that, for sufficiently small $r> 0$, we have $\xi \leq 0$ on $\overline V_r \times \overline V_r$. Therefore, for every $x\in E(\omega,1,R)$,
\begin{equation*}
\Psi^+(x) - \psi(x) - \max_{\partial E(\omega,1,R)} |h| \leq k \varphi(x) + \delta \leq Cr^{\epsilon^2/2-4\kappa} + Cr^{\epsilon}.
\end{equation*}
Set $\kappa := \epsilon^2/16$ to obtain $$\Psi^+-\psi - \max_{\partial E(\omega,1,R)} |h|\leq Cr^\beta$$ with $\beta := \epsilon^2/4$.
Reversing the roles of $\psi$ and $\Psi^+$, and repeating the same argument, we obtain \eqref{diffnce}.

We conclude by showing that \eqref{ratioestt} follows from \eqref{diffnce}, \eqref{psic1alph2}  and an interpolation inequality, which is recorded in Lemma~\ref{lipest} below. Set $v := \psi - \Psi^+$ and $\Gamma := (B_4 \setminus B_2) \cap \partial C_\omega$.
 By the Hopf lemma,
\begin{equation*}
\Psi^+(x) \geq \frac{1}{C} \dist(x, \Gamma) \quad \mbox{in } E(\omega, 2, 4).
\end{equation*}
Fix a smooth domain $\Omega$ such that $E(\omega, 2, 4)\subseteq \Omega\subseteq E(\omega, 3/2, 6)$. Since $v=0$ on $\Gamma$, Lemma \ref{lipest} implies that, for any $x\in E(\omega, 2, 4)$,
\begin{equation*}
\bigg|\frac{\psi(x)}{\Psi^+(x)} - 1\bigg|=\frac{|v(x)|}{\Psi^+(x)}\le C\frac{|v(x)|}{\dist(x, \Gamma)}\le C\| v\|_{C^{0,1}(E(\omega, 2, 4))} \leq C \|v\|_{C^{1,\gamma}(\Omega)}^{\frac{1}{1+\gamma}}
\|v\|_{L^\infty(\Omega)}^{\frac{\gamma}{1+\gamma}}.
\end{equation*}
The bounds \eqref{Psic1alph}, \eqref{psic1alph2} and \eqref{diffnce} yield \eqref{ratioestt} with $\bar\beta := \frac{\gamma}{1+\gamma}$.
\end{proof}

The following interpolation inequality is certainly known. We cannot find a reference, so we record a simple proof.

\begin{lem}\label{lipest}
Let $\Omega$ be a bounded $C^2$-domain. Then for each $u\in C^{1,\gamma}(\Omega)$,
$$
\|u\|_{C^{0,1}(\Omega)}\le C \|u\|_{C^{1,\gamma}(\Omega)}^{\frac{1}{1+\gamma}}
\|u\|_{L^\infty(\Omega)}^{\frac{\gamma}{1+\gamma}},
$$
for a constant $C$ depending only on $n,\gamma$, and the curvature of $\partial \Omega$.
\end{lem}
\begin{proof}
Let $\Omega$ satisfy a uniform interior ball condition with radius $h_0 > 0$. For each point $x\in \Omega$ we can fix a direction $\nu_x\in S^{n-1}$ such that $$|\partial_{\nu_x}u(x)|\ge (1/\sqrt{n})|Du(x)|.$$ We may also assume that $x+h\nu_x\in \Omega$ for all $h\in (0,h_0)$, since if necessary we can replace $\nu_x$ by $-\nu_x$.

Set $K_1:= \|u\|_{C^{1,\gamma}(\Omega)}$ and
$K_2:=\|u\|_{L^\infty(\Omega)}$. By Taylor's formula we have
\begin{equation}\label{inter1}
\frac{1}{\sqrt{n}}|Du(x)|h -K_1 h^{1+\gamma} \le |u(x+h\nu_x)-u(x)|\le 2K_2
\end{equation}
for any $h\in (0,h_0)$. We apply \eqref{inter1} with
$$h:=\frac{h_0}{2} |Du(x)|^{\frac{1}{\gamma}} K_1^{-\frac{1}{\gamma}}\le \frac{h_0}{2} $$
and, making the additional requirement that $h_0\le 2^{1-1/\gamma} n^{-2/\gamma}$, obtain after a  rearrangement of the resulting inequality that
\begin{equation*}
|Du(x)|\le CK_1^{\frac{1}{1+\gamma}}
K_2^{\frac{\gamma}{1+\gamma}}. \qedhere
\end{equation*}
\end{proof}
\bigskip

We next recall the Alexandrov-Bakelman-Pucci inequality with explicit dependence on the coefficients in the equation and the diameter of the domain, which was also used in the proof of Proposition~\ref{CDE} above.

\begin{prop}\label{abp} Let $\Omega$ be a bounded domain such that $\Omega\subset B_R$ for $R>0$, and set $\Omega_R:=\{x\::\:Rx\in \Omega\}\subset B_1$. There exist constants $C_1,C_2$ depending only on $n,\lambda,\Lambda$ such that if $u\in C(\overline{\Omega})$ is a solution of
$$
\pucci(D^2u)-\mu(x)|Du|\le f(x)\quad\mbox{in} \ \Omega$$
for some $f,\mu\in L^n(\Omega)$ with $f,\mu\ge 0$, then
\begin{equation}\label{abp1}
\sup_\Omega u \le \sup_{\partial\Omega} u + C_A R^2\|f(Rx)\|_{L^n(\Omega_R)},
\end{equation}
where
$$
C_A:=C_1\exp\left(C_2R\|\mu(Rx)\|_{L^n(\Omega_R)}\right).
$$
If in addition $\nu\in L^n(\Omega)$, $\nu\ge 0$ in $\Omega$ is such that $R^2\|\nu(Rx)\|_{L^n(\Omega_R)}<C_A^{-1}$ and
$$
\pucci(D^2u)-\mu(x)|Du|-\nu(x)|u|\le f(x)\quad\mbox{in}\ \Omega,$$
then
\begin{equation}\label{abp2}
\sup_\Omega u \le \left( 1-C_AR^2\|\nu(Rx)\|_{L^n(\Omega_R)}\right)^{-1}\left(\sup_{\partial\Omega} u + C_A R^2\|f(Rx)\|_{L^n(\Omega_R)}\right).
\end{equation}
\end{prop}
\begin{proof} The first statement is a scaled version (with respect to $R$) of the classical estimate, for strong solutions. For viscosity solutions it is Proposition 2.8 in \cite{KS1}. The second statement follows from Proposition 3.4 in \cite{S} and its proof.
\end{proof}

We  use the precise form of the constants which appear in this inequality. In particular, if $\mu(x) = |x|^{-1}$, $R=r^{-\kappa}$ and $\Omega \subset B_R\setminus B_1$ we see that $$C_A= C_1\exp\left(C_2\sqrt[n]{-\log(r)}\right),$$ hence for any $\delta>0$ we have  $C_A \leq r^{-\delta}$ if $r>0$ is sufficiently small. In particular, we see that for the case $f=1$, \eqref{abp1} implies \eqref{rrr} for small $r>0$.

Furthermore, if for some $\epsilon>0$ we have $\nu(x)\le r^{\epsilon} |x|^{-2+\epsilon}$ for $1\le |x|\le R=r^{-\kappa}$, then we compute that
$$
C_AR^2\|\nu(Rx)\|_{L^n(\Omega_d)} \le C_A r^{\epsilon-\kappa}\le r^{\epsilon-\kappa-\delta}\to 0\quad\mbox{ as }\; r\to 0,$$
if we set $\kappa= \delta = \epsilon/4$, so in this case \eqref{abp2} is valid for sufficiently small $r>0$, and the term in the first parentheses in \eqref{abp2} is close to 1 for large $R=r^{-\kappa}$.

\medskip

We are interested in applying the above continuous dependence estimate to the family of operators
$$
G_r(M,p,z,x) =r^2G(r^{-2}M, r^{-1}p, z, rx),
$$
where $G$ is some fixed operator. A typical example of an operator $G$ to which Proposition \ref{CDE} applies is the general extremal operator
$$
G^\pm(D^2u,Du,u,x)= \mathcal{P}^\pm(D^2u)  \pm \frac{\mu}{|x|} |Du|\pm \frac{\nu}{|x|^{1-\epsilon_1}} |Du| \pm \frac{\alpha}{|x|^{2-\epsilon_2}} |u| \pm \frac{\beta}{|x|^{2-\epsilon_2}}  f(x),$$
for some $\mu,\nu,\alpha,\beta\in\mathbb{R}^+$, $\epsilon_1, \epsilon_2>0$, $f\in C(B_1)$. The extremal operators $F^\pm, \mathcal{L}^\pm$ which we defined in Section \ref{prelim} correspond to $\nu=\alpha=0$, resp. $\mu=0$, $\epsilon_1=1,\epsilon_2=2$, and $f=0$.

Let us check that Proposition \ref{CDE} can be applied to these operators.
 First, for $G^\pm_r$ the limit operator $G_0$ is precisely $F^\pm$ and \eqref{closeness} holds, as it is very easy to check. That the Dirichlet problem \eqref{psirdirp} for $G^\pm_r$
 has a unique solution for sufficiently small $r>0$ follows from the ABP inequality, the above remarks and the available existence and uniqueness results; see \cite{CKLS,JS}, and pages 592-595 of \cite{S}. The ABP inequality implies that the solution $\psi$ of \eqref{psirdirp} with $G_r=G^\pm_r$ is uniformly bounded in $E(\omega,1,r^{-\kappa})$, independently of $r$. Since $E(\omega,1,r^{-\kappa})$ satisfies an exterior sphere condition, a standard barrier argument gives the Lipschitz bound \eqref{psic1alph}. The global gradient H\"older estimate stated in Proposition \ref{ggh} yields, in each set $B(x_0,1/4)\cap E(\omega,1,r^{-\kappa})$ with $x_0\in E(\omega,3/2,6)$, that the $C^{1,\gamma}$-norm of $\psi$ is bounded by a constant independent of $x_0$. Therefore \eqref{psic1alph2} holds.

These remarks are valid for any operator $G(M,p,z,x)$ which is positively homogeneous in $(M,p,z)$ and is appropriately bounded between $G^-$ and $G^+$, so  Proposition \ref{CDE} holds for such operators provided that a limit operator $G_0$ exists.

\section{Classification of isolated boundary singularities} \label{ibsec}

\subsection{Hypotheses and the statement of the theorem} \label{ibsec1}
In this section we study the solutions of fully nonlinear, uniformly elliptic equations which are possibly singular near a conical boundary point.

Let $\Omega$ be a domain,  $0 \in \partial \Omega$, and   \eqref{zetaiden} holds.
We assume we are given an continuous  operator $G: \Sy\times\R^n \times (\bar\Omega\cap \bar B_1) \to \R$ which satisfies the structural conditions \eqref{ellip}, \eqref{homogen} and \eqref{regxh} (but not necessarily \eqref{scaling}) for all $M \in \Sy$, $p\in \rn$, and $x,y\in \bar \Omega\cap \bar B_1$. We set
\begin{equation} \label{gr}
G_r(M,p,x) : = r^2 G\left(r^{-2} M, r^{-1} p, rx\right) = G(M,rp,rx).
\end{equation}
Obviously $G_r$ satisfies \eqref{ellip}, \eqref{homogen} and \eqref{regxh} with constants independent of $r\le1$; and $G_r=G$ if $G$ satisfies \eqref{scaling}.
We recall that \eqref{ellip}-\eqref{regxh} ensure that $G_r$ satisfies the comparison principle and that the Dirichlet problem associated to $G_r$ is uniquely solvable in any bounded domain satisfying an uniform exterior cone condition; see the remarks at the end of the previous section.

We further assume that there exists an operator $G_0=F$ which satisfies all assumptions \eqref{ellip}-\eqref{scaling} and which is such that $G_r\to G_0$ in the sense of \eqref{closeness}, that is for some $\epsilon\in (0,1)$,
\begin{equation} \label{closeness1}
\left| G_r(M,p,x) - G_0(M,p,y) \right| \leq O(|x-y|^\epsilon)|p| +  r^\epsilon (1+|p|+|M|)
\end{equation}
for all $M\in\Sy$, $p\in \rn$, and all $x,y$ such that $1\le |x|,|y|\le r^{-\epsilon}$. Then, as we explained in the previous section, the continuous dependence estimate (Proposition~\ref{CDE}) applies to $G_r$ and $G_0$, with $\Psi^+$ and $\alpha^+$ being defined as in Theorem~\ref{fundycones} for the operator $G_0(M,p,x)$
with respect to the cone $\co\omega$. Recall we have normalized $\Psi^+$ according to \eqref{normmax}, and that \eqref{Psic1alph} holds.

We now state our main result concerning the behavior of a solution near the isolated boundary point $0$. The following theorem reduces to Theorem \ref{ibs0} if $G$ satisfies all hypotheses \eqref{ellip}-\eqref{scaling}.

\begin{thm}\label{ibs}
Suppose that $G$ and $\Omega$ are as above, $u \in C\!\left( (\bar\Omega \!\setminus \! \{ 0 \})\cap B_1 \right)$ solves
\begin{equation} \label{ibseq}
\left\{ \begin{aligned}
& G(D^2u,Du,x) =0  & \mbox{in} & \ \Omega \cap B_1, \\
& u = 0 & \mbox{on} & \ (\partial \Omega \!\setminus \! \{ 0 \})\cap B_1  ,
\end{aligned} \right.
\end{equation} and $u$ is bounded below on $\Omega \cap B_1$. Then the conclusion of Theorem \ref{ibs0} is valid.
\end{thm}

\subsection{Flattening the boundary}
We may use the diffeomorphism $\zeta$ to rewrite the equation \eqref{ibseq} on the domain $\co\omega \cap B_1$. Define $H:\Sy\times\R^n\times \R\times\left( \co\omega\cap B_1\right) \to \R$ by
\begin{equation*}
H(M,p,x) : = G\!\left( D\zeta(x)^t M D\zeta(x) + D^2\zeta(x) p, \ D\zeta(x)^tp, \ \zeta^{-1}(x) \right),
\end{equation*}
where the product $D^2\zeta(x) p$ is the matrix with entries $\sum_{k=1}^n \zeta^k_{x_ix_j} p_k$. It is very easy to check that $H$ satisfies \eqref{ellip}, \eqref{homogen} and \eqref{regxh}, with  possibly modified constants depending only on the $C^2$-norm of the diffeomorphism~$\zeta$.
By defining $v(x) : = u\!\left( \zeta^{-1}(x) \right)$, we  see that \eqref{ibseq} is equivalent to
\begin{equation} \label{ibseqG}
H(D^2v,Dv,v,x) = 0 \quad \mbox{ in } \;  \co\omega \cap B_1,
\end{equation}
and $v = 0$ on $(\partial \co\omega \!\setminus \! \{ 0 \}) \cap B_1$. This can be checked at once for a $C^2$ function $u$, and an analogous calculation can be performed with smooth test functions to confirm the equivalence in the viscosity sense.   It is also trivial to see that the rescaled operator $H_r$ converges to $H_0=G_0$ in the sense of \eqref{closeness1}, thanks to our hypotheses on~$\zeta$.

This reduces the proof of Theorem~\ref{ibs} to the case that $\Omega = \co\omega$. Indeed, $G$ and $H$ satisfy the same structural conditions, and the function $u$ will satisfy the conclusion of Theorem~\ref{ibs} if and only if $v$ satisfies the same assertions with $\Omega$ replaced by $\co\omega$. It therefore suffices to prove Theorem~\ref{ibs} in the case $\Omega = \co\omega$, which we assume in the rest of this section.

\subsection{Strategy of proof}\label{proofstrategy}
The proof of Theorem~\ref{ibs} is delicate and technical. However, the underlying idea can already be found in the argument used in Section~\ref{maxpri} in order to prove Theorem~\ref{uniq}, in particular, the combined use of blow-up, the global Harnack inequality for quotients and the monotonicity properties of the maps $r\mapsto q^+(r), \, Q^+(r)$ given in Lemmas~\ref{PLorange} and~\ref{PLpurple}.

The main difficulty we encounter when trying to adapt these ideas in order to prove Theorem~\ref{ibs} is that, while we can define $q^+$ and $Q^+$ in almost the same way, these functions do not have the same monotonicity properties due to the fact that $u$ and $\Psi^+$ are not solutions of the same equation. In particular, this leads to a difficulty in obtaining \eqref{cqQC} and the existence of limits of $q^+$ and $Q^+$.

On the other hand, the function $u$ is \emph{nearly} a solution of $F=0$ very close to the origin. To see this, let us rescale the equation by defining $u_r(x) : = u(rx)$ and observe that $u_r$ is a solution of
\begin{equation*}
G_r(D^2u_r,Du_r,x) = 0 \quad \mbox{in} \ \co\omega \cap B_{1/r},
\end{equation*}
where $G_r$ is  defined by \eqref{gr}. Notice that for $r> 0$ very small, the operator $G_r$ is very close to $F=G_0$, by \eqref{closeness1}.
We therefore expect the maps $r\mapsto q^\pm(r), \, Q^\pm(r)$ to be \emph{nearly} monotone, in an appropriate way to be determined, which would permit us to pass to the limit $r\to 0$, and complete the proof using a blow-up argument as in Section \ref{maxpri}.

To make this intuition precise, we use the continuous dependence estimate from the previous section in order to control the separation between rescaled solutions of $F=0$ and those of  $G_r=0$ in $\co\omega \cap B_1$. Precisely, we use Proposition \ref{CDE} to construct  solutions of $G_r = 0$ in  large slices of the cone, and these functions are very close to $\Psi^+$. These solutions can be more easily compared to $u_r(x)  = r^{\alpha^+} u(rx)$, since they solve the same equation as $u_r$. The estimate between the ratios \eqref{ratioestt} will then lead us to an inequality between $q^+(r)$ and $q^+(2r)$ (and another between $Q^+(r)$ and $Q^+(2r)$) which roughly asserts that, for small $r$, the map $q$ is nearly monotone in~$r$. These inequalities may be iterated to
see that $Q^+(r)$ and $q^+(r)$ have limits as $r\to 0$. We also need Lemma \ref{snap} and the global Harnack inequality to show that $Q^+(r)$ is bounded as $r\to 0$.

If the limit of $q^+(r)$ is positive, we have \eqref{cqQC} and obtain alternative (ii) similarly to the proof of Theorem \ref{uniq}. On the other hand, if $\lim_{r\to0}Q^+(r)= 0$ then the ``almost monotonicity" estimate for $Q^+$ is enough to see that  in fact we have an algebraic rate of convergence: $Q^+(r) \leq Cr^{\beta_0}$ for some $\beta_0>0$. A blow-up argument is then combined with the uniqueness Theorem \ref{uniq} in order to show that $\beta_0$ may be improved to $\alpha^+$, which implies that $u$ is bounded from above. With the latter information in hand, another blow-up argument yields alternative (i).

\subsection{Characterization of isolated boundary singularities}
Given a function $u$ defined on $\coo \cap B_1$, define the quantities
\begin{equation*}
q^+(r) : = \inf_{E(\omega,r,2r)} \frac{\max\{ 0, u \}}{\Psi^+} \quad \mbox{and} \quad Q^+(r) : = \sup_{E(\omega,r,2r)} \frac{\max\{ 0, u \}}{\Psi^+}.
\end{equation*}

We begin our proof of Theorem~\ref{ibs} by showing that  if $u$ is a nonnegative supersolution of \eqref{ibseq}, then $r\mapsto q^+(r)$ is nearly nondecreasing  on an interval $(0,r_0)$.

\begin{lem} \label{qnearlymo}
Suppose that $u\in C\!\left( \overline{\mathcal{C}}_\omega\cap B_1\! \setminus\! \{ 0 \} \right)$  satisfies $u\ge 0$ and
\begin{equation}
G(D^2u,Du,x) \geq 0 \quad \mbox{in}  \ \co\omega \cap B_1.\\
\end{equation}
Then for some $\beta_0, r_0 > 0$ we have
\begin{equation} \label{qnearlymoeq}
q^+(s) \leq \exp(r^{\beta_0}) q^+(r) \quad \mbox{for every} \ 0 < s < r \leq r_0.
\end{equation}
In particular, $r\mapsto q^+(r)$ is bounded on $(0 , r_0)$.
\end{lem}
\begin{proof}
The first step is to show that for some $\beta_0 > 0$,
\begin{equation} \label{qnearlymono}
q^+(2r) \geq \left( 1 - r^{\beta}\right) q^+(r) \quad \mbox{for sufficiently small} \ r>0.
\end{equation}
Denote $ u_r(x) : = r^{\alpha^+} u(rx)$. Then $ u_r$ is a solution of the inequality
\begin{equation*}
G_r(D^2 u_r,D u_r,x) \geq 0 \quad \mbox{in} \ \co\omega \cap B_{1/r}.
\end{equation*}
Let $\psi_r$ be the solution of \eqref{psirdirp} for $h = -\Psi^+$ on $\co\omega \cap \partial B_{r^{-\kappa}}$ and $h=0$ elsewhere on $\partial E(\omega,1,r^{-\kappa})$. Observe that $\max_{\partial E(\omega,1,r^{-\kappa})}|h|\le Cr^{\kappa\alpha^+}$.

By the homogeneity of $\Psi^+$ and the definition of $q^+(r)$ we obtain  $$ u_r(x) \geq q^+(r)  r^{\alpha^+}\Psi^+(rx) = q^+(r) \Psi^+ (x)= q^+(r) \psi_r\quad\mbox{ on }\; \co\omega\cap \partial B_1,$$ and $u_r \geq 0 = q^+(r)\psi_r$ on the rest of the boundary of $E(\omega,1,r^{-\kappa})$. Since the comparison principle holds for the operator $G_r$ in $E(\omega,1,r^{-\kappa})$,  using the estimate \eqref{ratioestt} we have
\begin{equation*}
u_r \geq q^+(r) \psi_r \geq (1-r^{\beta_1}) q^+(r) \Psi^+ = (1-r^{\beta_1}) q^+(r) r^{\alpha^+}\Psi^+(rx)\quad \mbox{in} \ E(\omega,2,4)
\end{equation*}
for sufficiently small $r> 0$, where $\beta_1= (\bar\beta/2)\min\{\beta,\kappa\alpha^+\}$. This establishes \eqref{qnearlymono}.

By induction on \eqref{qnearlymono}, for every $k\in \mathbb{N}$ and $r>0$ sufficiently small, we have
\begin{equation*}
q^+(2^{-k}r) \leq q^+(r) \prod_{j=1}^k \left( 1 - \left(\frac{r}{2^j}\right)^{\beta_1}  \right)^{-1}.
\end{equation*}
Let $0<s<r$. Then for some $k$ we have $2^{-(k+1)}r<s\le 2^{-k}r$, and by using the obvious inequality $q^+(s)\le \max\{q^+(2^{-(k+1)}r),q^+(2^{-k}r)\}$ we obtain
\begin{equation*}
q^+(s) \leq q^+(r) \prod_{j=1}^{\infty} \left( 1 - r^{\beta_1} \left(2^{-\beta_1}\right)^j \right)^{-1}.
\end{equation*}
To estimate this infinite product  we use the elementary inequality
\begin{equation*}
-\log(1-y) \leq 2y \quad \mbox{for every}  \ 0 < y \leq \frac12\,,
\end{equation*}
to obtain for each $\beta>0$
\begin{align*}
\log \prod_{j=1}^\infty \left( 1 - r^\beta \left(2^{-\beta}\right)^j \right)^{-1} = - \sum_{j=1}^\infty \log\left( 1 - r^\beta \left(2^{-\beta}\right)^j \right) \leq 2r^{\beta} \sum_{j=1}^\infty \left(2^{-\beta}\right)^j = C r^\beta.
\end{align*}
Setting $\beta_0= \beta_1/2$ we have $Cr^{\beta_1}\le r^{\beta_0}$  for every $r>0$ small enough.  Hence for all $0 <s < r\le r_0$ we get \eqref{qnearlymoeq}.
\end{proof}
\medskip

We would like to conclude from Lemma~\ref{qnearlymo} that $Q^+(r)$ is bounded on $(0,r_0)$ if $u$ is a solution of \eqref{ibseq} which is bounded below. Of course, if $u$ is nonnegative then the boundary Harnack inequality (Proposition~\ref{bhq}) and Lemma~\ref{qnearlymo} imply that
\begin{equation} \label{qmqp}
Q^+(r) \leq C q^+(r)\le C.
\end{equation}
However, obviously \eqref{qmqp} fails to hold if $u$ changes sign in $E(\omega,r,2r)$. We circumvent this difficulty with the observation that if $u$ is bounded below, then it may only change sign if $Q^+(r)$ is small.

\begin{lem} \label{pullup}
Let $u \in C(\overline{\mathcal C}_\omega \cap B_1\!\setminus \{ 0 \})$ be a solution of \eqref{ibseq} which is bounded below. Then there exists $\sigma > 0$ such that $0 < r < \tfrac14$ and $Q^+(r) \geq \sigma r^{\alpha^+}$ imply that $u > 0$ in $E(\omega,r,2r)$.
\end{lem}
\begin{proof} Let us suppose that $u \geq -M$ in $\coo\cap B_1$ for some $M >0$. Let $\tilde u_r$ be the solution of the boundary value problem
\begin{equation*}
\left\{
\begin{aligned}
& G_r(D^2\tilde u_r, D\tilde u_r, x) = 0 & \mbox{in}  & \ E(\omega,1/2,4), \\
& \tilde u_r(x) = \max\{u(rx),0\} & \mbox{for} & \ x\in \partial E(\omega,1/2,4).
\end{aligned}
\right.
\end{equation*}
We have $u(rx)\le \tilde u_r(x) \le u(rx)+M$ for $x\in \partial E(\omega,1/2,4)$ so by the maximum principle the same inequality holds in the whole $E(\omega,1/2,4)$.

 If $Q^+(r) \geq \sigma r^{\alpha^+}$, then by the homogeneity of $\Psi$ and the global Harnack inequality (Proposition~\ref{bhq}) we have
\begin{equation*}
\inf_{E(\omega,1,2)}\frac{\tilde u_r}{\Psi^+} \ge c_0 \sup_{E(\omega,1,2)}\frac{\tilde u_r}{\Psi^+}\ge  c_0 \sup_{E(\omega,r,2r)}\frac{u(x)}{\Psi^+(x/r)} \geq c_0 \sigma .
\end{equation*}
for some $c_0> 0$ which does not depend on $r$. Recalling \eqref{normmax} we obtain $$u(rx)\ge c\tilde u_r(x) -M\ge c\sigma - M\quad\mbox{ on }\;E(\omega^\prime,1,2),$$ and $u(rx)\ge -M$ on $\partial E(\omega,1/2,4)$.
By Lemma~\ref{snap} we see that $\sigma > 0$ may be taken large enough relative to $M$ to ensure that $u > 0$ in $E(\omega,r,2r)$.
\end{proof}

From the previous lemmas and discussion we deduce the following corollary.

\begin{cor} \label{Qpbnd}
Suppose that $u \in C(\overline{\mathcal C}_\omega \cap B_1\!\setminus \{ 0 \})$ is a solution of \eqref{ibseq} which is bounded below. Then the map $r\mapsto Q^+(r)$ is bounded on $(0,\tfrac14)$.
\end{cor}
\begin{proof} If $Q$ is unbounded, then in particular we may choose a sequence $r_j \searrow 0$ such that $Q^+(r_j)\ge \sigma r_j^{\alpha^+}$, so by the previous lemma $u$ is positive in each $E(\omega,r_j,2r_j)$. The maximum principle then yields that $u> 0$ in $E(\omega,r_j,2r_1)$, and letting $j\to \infty$ we obtain that $u> 0$ in $\coo\cap B_{r_1}$. We may therefore apply the boundary Harnack inequality and Lemma~\ref{qnearlymo} to obtain \eqref{qmqp} for $r\leq 2 r_1$, a contradiction.
\end{proof}

Using Corollary~\ref{Qpbnd}, we show that if $u \geq -M$ is a solution of the problem \eqref{ibseq}, then $r\mapsto Q^+(r)$ is almost nonincreasing for small $r$.

\begin{lem} \label{Qnearlymo}
Suppose that $u \in C(\overline{\mathcal C}_\omega \cap B_1\!\setminus \{ 0 \})$ is a solution of \eqref{ibseq} which is bounded below. Then there exist constants $\beta_0, r_0 > 0$ such that
\begin{equation} \label{Qnearlymoeq}
Q^+(r) \leq \exp(r^{\beta_0})\, Q^+(s) + r^{\beta_0} \quad \mbox{for every} \ 0 < s < r \leq r_0.
\end{equation}
\end{lem}
\begin{proof}
According to Corollary~\ref{Qpbnd}, we have that $u \leq k_0\Psi^+$ in $\coo \cap B_{1/4}$ for some $k_0 > 0$. As usual we denote $u_r(x) : = r^{\alpha^+}u(rx)$ for small $r$, so that $u_r\le k_0\Psi^+$ in $\coo \cap B_{1/(4r)}$. Fix $s\le r < 2s < 1/2$ and consider the quantity
\begin{equation*}
T : = \max\{ Q^+(s) , r^{\frac12\kappa\alpha^+}\},
\end{equation*}
where $0 < \kappa < 1$ is as in our continuous dependence estimate (Proposition \ref{CDE}). Let $\psi_r$ be the solution of the Dirichlet problem \eqref{psirdirp} with $h =  k_0 T^{-1} \Psi^+$ on $\co\omega \cap \partial B_{r^{-\kappa}}$ and $h=0$ elsewhere on $\partial E(\omega, 1, r^{-\kappa})$. Observe that
\begin{equation*}
0 \leq h \leq k r^{-\frac12\kappa\alpha^+} \Psi^+ \leq k r^{\frac12\kappa\alpha^+} \quad \mbox{on} \ \partial E(\omega, 1, r^{-\kappa}).
\end{equation*}
Thus Proposition \ref{CDE} furnishes constants $\beta,r_0> 0$ such that for every $0< r \leq r_0$,
\begin{equation} \label{ratiobit}
\left| \psi_r  / \Psi^+ - 1 \right| \leq Cr^\beta \quad \mbox{in} \ E(\omega,2,4).
\end{equation}
Observe that $u_r\le k_0\Psi^+\le Th \leq T \psi_r$ on $\co\omega\cap \partial B_{r^{-\kappa}}$ and  $u_r\le Q^+(s)\Psi^+\le T\psi_r$ on $\co\omega\cap \partial B_1$. Hence  $u_r\le T\psi_r$ on $\partial E(\omega,1,r^{-\kappa})$, so the comparison principle and \eqref{ratiobit} imply that
\begin{equation*}
u_r \leq T \psi_r \leq (1+r^\beta) T \Psi^+ \quad \mbox{in} \ E(\omega, 2, 4)
\end{equation*}
for some $\beta > 0$ and small $r> 0$. Taking $\beta<\frac12\kappa\alpha^+$ we have the inequality $T \leq Q(s) + r^\beta$, and so we deduce that for $0 < s \le r <2s$,
\begin{equation} \label{recursbit}
Q^+(2r) \leq (1+r^\beta) T \leq (1+r^\beta) Q^+(s) + Cr^\beta.
\end{equation}
By induction on \eqref{recursbit} with $r=s$, we find that for every positive integer $k\geq 1$
\begin{equation} \label{indbd}
Q^+(r) \leq Q^+(\frac{r}{2^k}) \prod_{j=1}^{k+1} \left( 1 + 2^{-\beta j} r^\beta \right) + Cr^\beta \sum_{j=1}^{k+1} 2^{-\beta j} \prod_{\ell=1}^{j-1} \left(1+2^{-\beta \ell} r^\beta \right).
\end{equation}
By using the inequality $\log(1+t)\le t$ for $t>0$, it is easy to check that
\begin{equation*}
\prod_{j=1}^\infty \left( 1 + 2^{-\beta j} r^\beta \right) \leq \exp(Cr^\beta),
\end{equation*}
so for each $k\in \mathbb{N}$ we have
\begin{equation} \label{indbd1}
Q^+(r) \leq \exp(Cr^\beta) Q^+(\frac{r}{2^k})+ Cr^\beta.
\end{equation}
Since each $s<r$ is such that $2^{-(k+1)}r<s\le 2^{-k}r$ for some $k\in \mathbb{N}$, by using \eqref{recursbit} once more we obtain \begin{equation} \label{indbd2}
Q^+\!\left(\frac{r}{2^k}\right) \leq \max\{Q(s),Q(2s)\}\leq (1+s^\beta) Q^+(s)+ Cs^\beta.
\end{equation}
Combining \eqref{indbd1} and \eqref{indbd2} and setting $\beta_0=\beta/2$, we obtain the inequality \eqref{Qnearlymoeq}
for sufficiently small $r>0$. \end{proof}

\begin{lem} \label{existlimits} Assume that $u \in C(\overline{\mathcal{C}}_\omega \cap B_1\!\setminus \!\{0\})$ is bounded below and satisfies \eqref{ibseq}. Then the limits $\lim_{r\to 0} Q^+(r)$  and $\lim_{r\to 0} q^+(r)$ exist.
\end{lem}
\begin{proof}
It is clear from Corollary~\ref{Qpbnd} and \eqref{Qnearlymoeq} that $a: = \lim_{r\to 0} Q^+(r)\ge0$ exists (note \eqref{Qnearlymoeq} implies that every two convergent subsequences of $Q^+(r)$ have the same limit as $r\to 0$). If $a>0$, then Lemma \ref{pullup} applies for sufficiently small $r$, and we deduce that $u$ is positive in a neighborhood of zero. In this case, we may apply Lemma \ref{qnearlymo} and \eqref{qmqp} to infer that $b:= \lim_{r\to 0} q^+(r)$ exists and $0 < b \leq a$. In the case that $a=0$, then clearly $\limsup_{r\to 0} q^+(r) \leq \lim_{r\to 0} Q^+(r) = 0$.
\end{proof}

With $a: = \lim_{r\to 0} Q^+(r)$ and $b:=\lim_{r\to 0} q^+(r)$ as above, what remains of the proof of Theorem~\ref{ibs} is to show:
\begin{itemize}
\item if $a=0$, then $u$ can be continuously extended by defining $u(0)= 0$; and
\item if $a> 0$, then $a=b$.
\end{itemize}
We prove these statements in the following two lemmas.

\begin{lem} \label{extcont}
Assume that $u \in C(\overline{\mathcal{C}}_\omega \cap B_1\!\setminus \!\{0\})$ is bounded below and satisfies \eqref{ibseq}, and suppose also that
\begin{equation}\label{Qtozero}
\lim_{r\to 0} Q^+(r) = 0.
\end{equation}
Then $u(x) \rightarrow 0$ as $|x| \to 0$.
\end{lem}
\begin{proof}
Sending $s\to 0$ in \eqref{Qnearlymoeq} we obtain the estimate
\begin{equation} \label{algrate}
Q^+(r) \leq  r^{\beta_0}
\end{equation}
for some $\beta_0> 0$ and sufficiently small $r >0$. We will combine a blow-up argument with Theorem~\ref{uniq} to improve this algebraic rate of convergence by replacing $\beta$ with $\alpha^+$ in \eqref{algrate}. That is, we claim that
\begin{equation}\label{claimche}
Q^+(r)\le Cr^{\alpha^+},
\end{equation}
for sufficiently small $r$ and some $C$ independent of $r$.
Note that \eqref{claimche} is equivalent to  $u$ being bounded in a neighborhood of the origin, by the definition of $Q^+(r)$ and the homogeneity of $\Psi^+$.

Suppose on the contrary that \eqref{claimche} does not hold. Then  Lemma \ref{pullup} applies and $u>0$ in some neighborhood of the origin  $\co\omega\cap B_{r_0}$.

  Denote $ M(s) : = \sup_{\co\omega\cap \partial B_{s}} u$. By the strong maximum principle, $ M(s)$ cannot have a  local maximum on $(0,r_0)$. Since $ M(s)$ is unbounded in any neighborhood of zero, we see that $s\mapsto  M(s)$ is nonincreasing on some interval $(0,r_0]$ and $M(s)\to \infty$ as $s\to 0$. Define the function
\begin{equation*}
\bar u_r(x) : =\frac{r^{\alpha^+} u(rx)}{ Q^+(r)} \quad\mbox{ in }\; \co\omega \cap B_{r_0/r}.
\end{equation*}
Obviously $\sup_{E(\omega,1,2)} \bar u_r \leq \sup_{E(\omega,1,2)} \Psi^+ \leq C$, and
\begin{equation}\label{supche}
\sup_{E(\omega,1,2)} \frac{\bar u_r}{\Psi^+} = 1.
\end{equation}
 Using the global Harnack inequality we see that $\bar u_r/\Psi^+\le C(A)$ in $E(\omega, 1/A,A)$ for any $A>0$ and $r<1/A$. Therefore $\bar u_r$ is bounded in any compact subset of $\overline{\mathcal{C}}_\omega \!\setminus \!\{ 0 \}$ so we may use the global H\"older estimates to find a subsequence $r_j \to 0$ and a function $v\ge0$ such that $\bar u_{r_j} \rightarrow v$ uniformly on bounded subsets of $\overline{\mathcal{C}}_\omega \!\setminus \!\{ 0 \}$.

We claim that necessarily $v \equiv \Psi^+$, from which we can conclude that the entire sequence $\bar u_r$ converges  to $\Psi^+$ on bounded subsets of $\coo$. Passing to limits in the viscosity sense along $r_j$ and using that $G_r \rightarrow F$ as $r\to 0$ (recall \eqref{gr}-\eqref{closeness1}), we see that $v$ satisfies
\begin{equation*}
F(D^2v, Dv,x) = 0 \quad \mbox{in} \ \co \omega,
\end{equation*}
as well as $v=0$ on $\partial \co\omega\!\setminus\! \{ 0\}$. We also have  $v > 0$ in $\coo$ by the strong maximum principle ($v\equiv0$ is excluded by \eqref{supche}).

Recalling that $M_r(s):=\max_{\co\omega\cap \partial B_s} \bar u_r$ is nondecreasing on $(0,r_0/r)$ we deduce that $s\mapsto \max_{\co\omega\cap \partial B_s} v$ is nonincreasing for all $s>0$, and in particular $v$ is bounded away from the origin. We may now conclude from Theorem~\ref{uniq} that $v \equiv t \Psi^+$ for some $t> 0$. Passing to limits in \eqref{supche} we get $t=1$.  The claim is proved.

We have shown that the full sequence $\bar u_r \rightarrow \Psi^+$ locally uniformly in $\overline{\mathcal{C}}_\omega \!\setminus \!\{0\}$ as $r \to 0$. This implies that for each $\ep > 0$, there exists $r_\ep > 0$ such that
\begin{equation*}
\frac{ M(2r)}{ M(r)} = \frac{\sup_{\co\omega \cap \partial B_2} \bar u_r}{\sup_{\co\omega \cap \partial B_1} \bar u_r} \leq 2^{-\alpha^+ + \ep}  \quad \mbox{for} \ 0 < r \leq r_\ep,
\end{equation*}
where, as usual, $M(r) = \sup_{\co\omega\cap \partial B_r} u$. By induction, we then have
\begin{equation}\label{fre1}
M(2^{-k}r_0) \geq 2^{(\alpha^+-\varepsilon)k}M(r_0)
\end{equation}
for each $k\in \mathbb{N}$.  On the other hand, obviously,
\begin{equation}\label{fre2}
M(r) = \sup_{\co\omega \cap \partial B_r}[(u/\Psi^+)\Psi^+]\leq  Q^+(r)r^{-\alpha^+}
\end{equation}
 If we now fix $\ep = \beta_0/2$, combining \eqref{algrate}, \eqref{fre1} and \eqref{fre2} with $r = 2^{-k}r_0$ yields a contradiction for sufficiently large $k$. We have proved \eqref{claimche}.

Knowing that $u$ is bounded in a neighborhood of the origin, the third and final step is to show that in fact $u(x) \to 0$ as $|x| \to 0$. To this aim we will use a simple blow-up argument combined with the strong maximum principle and the stability of viscosity solutions with respect to uniform convergence.

Let us give the argument, for completeness.  Set  $M_0= \limsup_{x\to 0}u(x)$ and $m_0= \liminf_{x\to 0}u(x)$. Assume $ M_0>0$ and take a sequence $r_j\to 0$ such that $M(r_j)\to M_0$ as $j\to\infty$. Set $u_j(x)=u(r_jx)$. Then $G_{r_j}(D^2u_j, Du_j,x) =0$ and $u_j$ is bounded independently of $j$ in $E(\omega, 1/2, 4)$. By the H\"older estimates we can pass to the limit and obtain a function $\bar u$ such that $F(D^2\bar u, D\bar u, x)=0$  and $m_0\le \bar u \le M_0$ in $E(\omega, 1/2, 4)$, and $\bar u(\bar x)=M_0>0$ for some point $\bar x\in \partial B_1$. The strong maximum principle then implies $\bar u\equiv M_0$ in $E(\omega, 1/2, 4)$ which contradicts the boundary condition $\bar u=0$ on $\partial \co\omega$.

 In the same way, we see that $m_0<0$ is impossible. Thus $u(x) \to 0$ as $|x| \to 0$.
\end{proof}

We next discuss the case that $\lim_{r\to 0} Q^+(r) > 0$, and show that the solution $u$ must blow up to a multiple of $\Psi^+$.

\begin{lem}
Suppose $u \in C(\overline{\mathcal{C}}_\omega \cap B_1\!\setminus \!\{0\})$ is a solution of \eqref{ibseq} which is bounded below, and suppose also that
\begin{equation*}
a: = \lim_{r\to 0} Q^+(r) > 0.
\end{equation*}
Then $\lim_{r\to 0} q^+(r) = a$. That is,
\begin{equation}
\lim_{\co\omega \ni x\to 0}\frac{ u(x)}{\Psi^+(x)} = a .
\end{equation}
\end{lem}
\begin{proof}
As we discussed in the proof of Lemma~\ref{existlimits}, under our hypotheses we have that the limit $b: = \lim_{r\to 0} q^+(r)$ exists and $0 < b \leq a$. Using a blow-up argument, we may deduce that $b=a$. In fact, the argument is identical to the one used in the  proof of Theorem 2, so we omit it.
\end{proof}

\begin{proof}[{Proof of Theorem~\ref{ibs}}]
The proof is easily assembled from the above lemmas.
\end{proof}

\begin{proof}[{Proof of Theorem~\ref{poisson}}]
Fix $x_0 \in \Omega$.  For each small $ r \in(0, |x_0|)$, let $u_r$ be the solution of the Dirichlet problem
\begin{equation*}
\left\{ \begin{aligned}
& F(D^2u_r,Du_r,x) = 0 & \mbox{in} & \ \Omega \setminus B_r, \\
& u_r = g_r & \mbox{on} & \ \partial (\Omega \setminus B_r),
\end{aligned} \right.
\end{equation*}
where $g_r$ is a nonnegative continuous function on $\partial (\Omega \setminus B_r)$, $g_r\not\equiv 0$, which is supported on $\partial B_r$. By multiplying $u_r$ by a positive constant, we may suppose that $u_r(x_0) = 1$. By the Harnack inequality and global H\"older estimates, we may select a subsequence $r_j \to 0$ such that $u_{r_j}$ converges locally uniformly in $\overline\Omega\setminus \{ 0 \}$ as $j\to \infty$ to a continuous function $u_0$, which satisfies, by the stability of viscosity solutions,
\begin{equation} \label{u0poiss}
\left\{ \begin{aligned}
& F(D^2u_0,Du_0,x) = 0 & \mbox{in} & \ \Omega \setminus \{ 0 \}, \\
& u_0 = 0 & \mbox{on} & \ \partial\Omega \setminus \{ 0 \}.
\end{aligned} \right.
\end{equation}
Since $u_0(x_0) = \lim_{j\to\infty} u_{r_j} (x_0) = 1$, the strong maximum principle implies that $u_0 > 0$ in $\Omega$. The alternative (i) in Theorem~\ref{ibs}  is excluded by the maximum principle, since $\Omega$ is bounded. Hence $u_0 (x)/ \Phi^+(\zeta(x)) \rightarrow t_0$ as $x \to 0$, for some $t_0> 0$.

Let $u\not\equiv 0$ be another nonnegative solution of \eqref{u0poiss}. By the strong maximum principle, $u > 0$ in $\Omega$. Theorem~\ref{ibs} and the maximum principle imply that $u(x) / \Phi^+(\zeta(x)) \rightarrow s$ as $x\to 0$, for some $s> 0$, hence $u(x) / u_0(x) \rightarrow s/t_0 : = t$ as $x \to 0$. Therefore for each $\ep>0$ and $r>0$ sufficiently small,
\begin{equation*}
(t-\ep)u_0\le u\le (t+\ep)u_0 \quad \mbox{on} \ \Omega \cap \bar B_r,
\end{equation*}
and $u=u_0=0$ on $\partial \Omega\setminus B_r$. By the maximum principle,
\begin{equation*}
(t-\ep)u_0\le u\le (t+\ep)u_0 \quad \mbox{in} \ \Omega \setminus B_r.
\end{equation*}
Sending $r\to0$ and then $\ep\to 0$ yields $u\equiv tu_0$.
\end{proof}

\subsection{Behavior of a solution near a conical boundary point}\label{behdir}

Simple modifications of the method we have used to prove Theorem~\ref{ibs} can be made to obtain Theorem \ref{exthopf} and the assertions in Remark \ref{remk}. We study the quantities
\begin{equation*}
q^-(r) : = \inf_{E(\omega,r,2r)} \frac{\max\{ 0, u \}}{\Psi^-} \quad \mbox{and} \quad Q^-(r) : = \sup_{E(\omega,r,2r)} \frac{\max\{ 0, u \}}{\Psi^-},
\end{equation*}
using similar ideas as above. The analysis turns out to be easier than that of $q^+(r)$ and $Q^+(r)$, since $q^-(r)$ and $Q^-(r)$ have the reverse monotonicity; that is, $q^-(r)$ is (nearly) nonincreasing, while  $Q^-(r)$ is (nearly) nondecreasing. Hence these quantities cannot tend to zero or to infinity, and so we do not need to use the global Harnack inequality.

The use of continuous dependence estimate can also be considerably simplified when we deal with a solution defined up to the boundary, in particular in the model cases. For instance if $F=F(D^2u)$ we solve the approximate problem $G_r(D^2\psi_r, D\psi_r,  x)=0$ in $\co\omega\cap B_1$, $\psi_r=\Psi^-$ on $\partial(\co\omega\cap B_1)$. Then an easier  argument than the one we gave in the proof of Proposition \ref{CDE} yields that $\Psi_r^-/\Psi^-\to 1$ on compact subsets of $\co\omega\cap B_{1}$, with an algebraic rate which we can estimate. In the general case an exact analogue of the continuous dependence estimate is obtained, with nearly the same proof, if in Proposition \ref{CDE} we replace  $\Psi^+$ by $\Psi^-$, $E(\omega, 1,r^{-\kappa})$ by $E(\omega, r^\kappa, 1)$ and $E(\omega, 2,4)$ by $E(\omega, 1/4, 1/2)$.

\begin{proof}[Proof of Theorem \ref{exthopf}] If $\psi_r$ is the function given by Proposition \ref{CDE} thus modified with $h=0$ on $\co\omega\cap\partial B_1$, $h=0$ on $\partial\co\omega$, $h=-\Psi^-$ on $\co\omega\cap B_{r^\kappa}$, and $u>0$ is as in Theorem~\ref{exthopf} (possibly with $F$ replaced by the more general $G$ as in Section \ref{ibsec1}), by flattening the boundary and by setting $q^-(r) = \inf_{\co\omega\cap(B_1\setminus B_{1/2})} (u_r/\Psi^-)$ where $u_r(x) := r^{\alpha^-}u(rx)$, exactly as in the proof of Lemma \ref{qnearlymo} we see that the maximum principle implies $u\ge q^-(r)\psi_r$ in $E(\omega, r^\kappa, 1)$, and hence
$$
q^-(r/2)\ge q^-(r)(1-r^\beta)
$$
for all sufficiently small $r>0$. Iterating this inequality we obtain, as before,
$$
q^-(r)\le e^{r^\beta}q^-(s),\quad\mbox{ for all }\; 0<s<r\le r_0,
$$
which implies that $q^-(s)\ge c_0q^-(r_0)>0$ for all $s>0$ small. Letting $s\to 0$ we obtain the statement of Theorem \ref{exthopf}. Note that $q^-(r_0)>0$ by the standard Lipschitz estimates (Theorem \ref{ggh}) and the classical  Hopf's lemma applied at the flat lateral boundary of the domain $E(\omega, r_0/4, 2r_0)$. \end{proof}

Let us also sketch the proof of the statement we mentioned in Remark \ref{remk}.  If $u>0$ is such that $u$ solves \eqref{isol} (possibly with $F$ replaced by  $G$) and $u(0)=0$, we proceed exactly like in the proof of Lemma \ref{Qnearlymo} to infer that
$$
u_r \le \max\{Q^-(r), r^{-\alpha^-\kappa/2}\}\psi_r$$
which after a iteration yields
$$
Q^-(s)\le e^{r^\beta}Q^-(r)+r^\beta,\quad\mbox{ for all }\; 0<s<r\le r_0,
$$
so $Q^-(s)$ is bounded and has a limit as $s\to 0$. By what we already proved for $q^-(s)$ we see that it also has a positive limit as $s\to 0$. These two limits have to coincide, by the same blow-up argument as the one we used in the proof of Theorem \ref{uniq}.

\section{The \pl principles}\label{sectpoisson}

As a further application of Theorem~\ref{ibs}, we prove the \pl principles stated in the introduction.

\begin{proof}[Proof of Theorem \ref{PLloc}]
In Section~ \ref{maxpri} we proved the result in the case that $\Omega = \coo$. We need only show how to modify the argument for a general domain $\Omega$ for which $0\in \partial \Omega$ and \eqref{zetaiden} holds, and for a subsolution of
$$
G(D^2u,Du,x)\le 0,$$ where \eqref{gr}-\eqref{closeness1} are verified for $G$ (resp. for $G^*$, if we have a condition at infinity), with $G_0=F$. We first flatten the boundary as in the previous section, so that the function $u_s(x)= s^{\alpha^+}u(sx)$ is a solution of $G_s(D^2u_s, Du_s, x)\le 0$ in $\co\omega\cap B_{1/s}$. We consider the function $\psi_s$ given by Proposition~\ref{CDE} with $h\equiv0$  (we apply the proposition with $E(\omega,1,r^{-\kappa})$ replaced by $E(\omega,1/2,r^{-\kappa})$ and $E(\omega, 2, 4)$ replaced by $E(\omega, 1, 2)$) and repeat the  argument given in Section~\ref{maxpri} with $\Psi^+$ replaced by $\psi_s$, observing that $\psi_s\ge \tfrac12 \Psi^+$ in $E(\omega, 1, 2)$ for  small $s> 0$.
\end{proof}

\begin{proof}[Proof of Theorem~\ref{PL}]
We may assume that $\mathcal{D}^\prime = \mathcal{D}$, by replacing $u$ by $u^+$ and then defining $u$ to be zero in $\mathcal{D}\setminus\mathcal{D}^\prime$. Theorem \ref{PLloc} implies that, for each $\ep>0$ and sufficiently small $r>0$, we have $u\le \ep$ on $\mathcal{D}\cap(\partial B_r\cup\partial B_{1/r})$. The maximum principle implies that $u\leq \ep$ in $\mathcal{D} \cap (B_{1/r} \setminus B_r)$ for small $0 < r <1$. Sending $r\to \infty$ and then $\ep \to 0$ yields the result.
\end{proof}

\begin{proof}[Proof of Theorem~\ref{PL2}]
The result follows at once from Theorem \ref{PLloc}, after we note that $\widetilde F(D^2u, Du, x)\ge 0$ is equivalent to $F(D^2(-u), D(-u), x)\le 0$.
\end{proof}

We can prove a sharper \pl maximum principle, of Nevanlinna-Heins type (see the references in \cite{Se}) in the case that the domain $\Omega$ has only one conical singularity (at zero or infinity), and $u\le 0$ on the whole boundary $\partial \Omega$ except possibly at the singularity. In this classical case stronger results are readily obtained, as we will see in Proposition \ref{PLNH} below. To prove this proposition we will combine the method of Gilbarg, Hopf and Serrin with the following lemma, in which  we construct a special barrier function,  with the help of Theorem \ref{ibs0}.

\begin{lem} \label{barrier}
Let $\Omega$ be a bounded Lipschitz domain such that $0\in \partial \Omega$ and \eqref{zetaiden} holds, and fix $x_0 \in \Omega$. Let $G$ be an operator as in Section~\ref{ibsec1}. For each $r > 0$ sufficiently small, there exists a function $\phi_r \in C(\overline\Omega\setminus B_r)$ and a constant $a_r$ which satisfies
\begin{equation}\label{boundi}
cr^{-\alpha^+} \leq a_r \leq Cr^{-\alpha^+},
\end{equation}
such that $\phi_r(x_0) = 1$ and
\begin{equation}\label{eqii}
\left\{ \begin{aligned}
& G(D^2\phi_r, D\phi_r,x) =0 & \mbox{in} & \ \Omega \setminus B_r, \\
& 0 \leq \phi_r \leq a_r & \mbox{on} & \ \partial \Omega,\\
& \phi_r = 0 & \mbox{on} & \ \partial\Omega \setminus B_{2r}, \\
& \phi_r = a_r & \mbox{on} & \ \Omega \cap \partial B_{r},
\end{aligned} \right.
\end{equation}
\end{lem}
\begin{proof}
The existence of a function $\phi_r$ for some $a_r$ is simple to obtain: we solve \eqref{eqii} with $a_r=1$ and then divide the function thus obtained by its value at $x_0$. Moreover, by the H\"older estimates and Theorem~\ref{poisson}, the function $\phi_r$ converges locally uniformly on $\overline \Omega \setminus \{ 0\}$ to the unique nonnegative function $\phi_0\in C(\overline \Omega \setminus \{ 0 \})$ which satisfies $\phi_0(x_0) = 1$ and
\begin{equation*}
\left\{ \begin{aligned}
& G(D^2\phi_0, D\phi_0,x) =0 & \mbox{in} & \ \Omega, \\
& \phi_0 = 0 & \mbox{on} & \ \partial \Omega \setminus \{ 0 \}.
\end{aligned} \right.
\end{equation*}
By Theorem~\ref{ibs}, there exists $t> 0$ such that
\begin{equation}\label{convi}
\phi_0(x)/\Psi^+(\zeta(x))\to t>0\quad\mbox{as}\ x\to 0, \,x\in \Omega.
\end{equation}

It remains to obtain the bounds on $a_r$.  Set $\delta_r= r^{\alpha^+}a_r$. By \eqref{convi}, for each $\ep>0$ and sufficiently small $r$ we have $\phi_0\le (t+\ep)r^{-\alpha^+}=(t+\ep)\delta_r^{-1}\phi_r$ on $\Omega\cap\partial B_r$ and $\phi_0=0$ on the rest of the boundary of $\Omega_r$, so by the maximum principle $\phi_0\le (t+\ep)\delta_r^{-1}\phi_r$ in $\Omega_r$. Evaluating at $x_0$ we get the second inequality in \eqref{boundi}.

Set $\tilde\phi_r = r^{\alpha^+}\phi_r(rx)$ and $\phi_{0r}=r^{\alpha^+}\phi_0(rx)$. Observe that $\tilde\phi_r\le \delta_r$ in $(\Omega/r)\cap(B_8\setminus B_1)$. Then  \eqref{convi} and the global Harnack inequality imply $\tilde\phi_r/\phi_{0r}\le C\delta_r$ in $(\Omega/r)\cap(B_4\setminus B_2)$. By the maximum principle $\tilde\phi_r\le C\delta_r\phi_{0r}$ in $(\Omega/r)\setminus B_2$. Evaluating at $x_0/r$ we obtain the first inequality in \eqref{boundi}.
\end{proof}

\begin{prop}\label{PLNH}
Suppose that $\Omega$ is  such that $0\in \partial \Omega$ and \eqref{zetaiden} holds.  Let $\mathcal{D}\subseteq\Omega$ be a bounded domain, $x_0\in \mathcal{D}$, and $\phi_0$  be the unique solution of \eqref{dirpoisson} in  $\Omega \cap B_R$ such that $\phi_0(x_0) = 1$,  where $R> 0$ is large enough that $\mathcal D \subset B_R$. Let  $u$ satisfy the inequality
\begin{equation*}
F(D^2u, Du, x) \leq 0 \; \mbox{ in }\;   \mathcal{D}, \quad u\le 0 \;\mbox{ on }\; \partial \mathcal{D}\setminus\{0\}\qquad (\mbox{resp. in }\; \Omega^*, \;\mbox{ on } \partial \Omega^*) .
\end{equation*}
Set $M(r) := \max_{ \mathcal{D}\cap \partial B_r}{u}$, and assume in addition that
\begin{equation}\label{condformpliminf}
\liminf_{r\to0}\; r^{\alpha^+}M(r) \le k \qquad \Big( \mbox{resp.}\quad  \liminf_{R\to\infty}\; R^{\alpha^-}M(R)  \le l\Big)
\end{equation}
for some $0\le k, l <\infty$. Then there exists a constant $A>0$, depending only on $F$, $\omega$, and $x_0$, such that
$$
u\le Ak\phi_0 \; \mbox{ in }\;   \mathcal{D}\qquad (\mbox{resp. }\; u\le Al\phi_0 \mbox{ in }\;\Omega^*).
$$
\end{prop}

\begin{proof} Take sequences $r_j\to 0$, $k_j\searrow k$ such that $u\le k_jr_j^{-\alpha^+}$ on $\partial B_{r_j}$. Let $\phi_j = \phi_{r_j}$ be the functions constructed in Lemma \ref{barrier}. By \eqref{boundi}, we have $u\le Ak_j \phi_j$ on $\partial B_{r_j}$ and $u\le 0 \le Ak_j\phi_j$ on $\partial \mathcal{D}\setminus B_{r_j}$. Hence, for each $x\in \mathcal{D}$,
\begin{equation*}
u(x)\le Ak_j \phi_j(x) \to Ak\phi_0(x)\quad\mbox{as} \quad j\to \infty. \qedhere
\end{equation*}
\end{proof}

\appendix

\section{Harnack inequalities up to the boundary} \label{BHQproof}

Here we prove Proposition~\ref{bhq}, the global Harnack inequality for quotients of positive solutions. First we recall the interior Harnack inequality, a proof of which can be found in \cite[Theorem~3.6]{QS} and the references therein.

\begin{prop} \label{hq}
Assume that $\Omega$ is a bounded domain and $u\in C(\Omega)$ is a positive solution of the inequalities
\begin{equation}\label{harcon1}
\mathcal{L}^+[u] \,\geq\, 0\, \geq\,   \mathcal{L}^-[u]\quad\mbox{in}\quad \Omega.
\end{equation}
Then for each $\Omega' \subset \subset \Omega$ we have the estimate
\begin{equation}\label{harq1}
\sup_{\Omega'} u \leq C \inf_{\Omega'} v.
\end{equation}
The constant $C>1$ depends only on $n$, $\lambda$, $\Lambda$, $\mu$, $\nu$, and the number of balls of radius $\frac{1}{2}\dist(\Omega',\partial \Omega)$ needed to cover $\Omega'$.
\end{prop}

Another preliminary result is the following boundary Harnack inequality.

\begin{prop}\label{bhq1} Assume $u,v$ are positive solutions of the inequalities
\begin{equation}\label{harcon}
  \mathcal{L}^+[u] \,\geq\, 0\, \geq\,   \mathcal{L}^-[u]\,, \qquad \mathcal{L}^+[v] \,\geq\, 0\, \geq\,   \mathcal{L}^-[v]\quad\mbox{in}\quad \Omega.
\end{equation}
Let  $Q_{2}= Q_2(z_0)\subset\Omega$ be a cube with side $2$ centered at $z_0=(1,0,\ldots,0)\in \Omega$, and assume $u=v=0$ on $Q_2\cap \{x_1=0\}$. Then there exists a constant $C=C(\lambda,\Lambda,\mu,\delta)>0$ such that for each $x\in Q_1 = Q_1(z_0/2)$ we have
\begin{equation}\label{harres}
   \frac{1}{C}\frac{u(x)}{u(z_0)}\le \frac{v(x)}{v(z_0)}\le C\frac{u(x)}{u(z_0)}\,.
\end{equation}
\end{prop}
\begin{proof}
We can assume that the principal eigenvalues of the operators $\mathcal{L}^+[u]$, $\mathcal{L}^-[u]$ are positive in $Q_2$, that is, these operators satisfiy the comparison principle in each subdomain of $Q_2$ (see \cite{QS}). Indeed,  we can find $r_0=r_0(\lambda,\Lambda,\mu,\delta)$ sufficiently small that these eigenvalues  be positive in each cube with side $r_0$, then apply the result in $[2/r_0]$ cubes with base on $\{x_1=0\}$ included in $Q_2$, and conclude with the help of the interior Harnack inequality.

Repeating the proof of the  Harnack inequality in \cite[Section 2]{BCN1}, with obvious modifications (replacing the operator $M$ in theorems 1.3-1.5 in that paper by the Pucci extremal operators and so forth), we infer that
\begin{equation}\label{harbcn}
u(x)\le Cu(z_0)\qquad\mbox{ for }\mbox{ each }\;x\in Q_2.
\end{equation}

Next, using some ideas from \cite{CFMS}, we take two functions $g_1,g_2\in C(\partial Q_2)$  such that $0\le g_1,g_2\le 1$ on $\partial Q_2$; $g_1\equiv1$ on $\partial Q_2\cap\{x_1\ge 1/2\}$ and $g_1\equiv 0$ on $\partial Q_2\cap\{x_1\le 1/4\}$; and $g_2\equiv1$ on $\partial Q_2\cap\{x_1>0\}$, while $g_2\equiv0$ on $\partial Q_2\cap\{x_1=0\}\cap\{|x|\le \sqrt{2}\}$. Let $w_1,w_2\in W^{2,p}_\mathrm{loc}(Q_2)\cap C(\overline{Q_2})$, $p<\infty$, be the solutions of the Dirichlet problems
\begin{equation*}
\left\{
\begin{array}{rcl}
\mathcal{L}^+[w_1] &=&0 \quad\mbox{ in }Q_2\\
w_1&=&g_1\;\;\mbox{ on }\partial Q_2,
\end{array}
\right.
\qquad\qquad
\left\{
\begin{array}{rcl}
\mathcal{L}^-[w_2] &=&0 \quad\mbox{ in }Q_2\\
w_2&=&g_2\;\;\mbox{ on }\partial Q_2.
\end{array}
\right.
\end{equation*}

Hopf's lemma and the boundary Lipschitz estimates imply that
\begin{equation}\label{ineqww}
w_2(x)\le C\mathrm{dist}(x,\partial Q_2)\le C w_1(x)\qquad\mbox{for }\mbox{ each }\; x\in Q_1.
\end{equation}

On the other hand, we clearly have $u\le Cu(z_0)w_2$ on $\partial Q_2$, by \eqref{harbcn} and the definition of $w_2$. Hence the comparison principle implies
\begin{equation}\label{inequw}
u(x)\le  Cu(z_0) w_2(x)\qquad\mbox{for }\mbox{ each }\; x\in Q_2.
\end{equation}

Finally, the interior Harnack inequality implies that $v(x)\ge C v(z_0)$, for each $x\in \partial Q_2\cap\{x_1\ge 1/6\}$. This implies that $v\ge Cv(z_0)w_1$ on $\partial Q_2$, so by the comparison principle
\begin{equation}\label{ineqvw}
v(x)\ge  Cv(z_0) w_1(x)\qquad\mbox{for }\mbox{ each }\; x\in Q_2.
\end{equation}

Combining \eqref{ineqww}, \eqref{inequw}, \eqref{ineqvw} yields the statement of the proposition.  \end{proof}

\begin{proof}[{Proof of Proposition~\ref{bhq}}] 
We denote $\Omega_{0} : = \{x\in \Omega : \mathrm{dist}(x,\partial\Omega)>r_0\}$, and $\Omega'' : = \Omega' \cap \Omega_{0}$, where $r_0< \frac{1}{2}\dist(\Omega',\partial \Omega\setminus\Sigma)$ is fixed so small that $\partial \Omega_0$ is $C^2$-smooth. By straightening the boundary (see, e.g., the proof of \cite[Theorem~1.4]{BCN1}), we obtain, by Propositions \ref{hq} and \ref{bhq1},
\begin{equation*}
\max\left\{ \sup_{\Omega''}\frac{u}{v}, \; \sup_{\Omega'\setminus\Omega''}\frac{u}{v}\right\}  \leq \sup_{\Omega'\cap\partial\Omega_{0}}\frac{u}{v}  \leq \inf_{\Omega'\cap\partial\Omega_{0}}\frac{u}{v}\leq \min\left\{ \inf_{\Omega''}\frac{u}{v},\; \inf_{\Omega'\setminus\Omega''}\frac{u}{v}\right\}. \qedhere
\end{equation*}
\end{proof}

\section{Proofs of Lemma \ref{alphlowbnd} and \ref{alphuppbnd}} \label{hell}

We give the precise computations which establish that the functions defined in \eqref{super1} and \eqref{sub1} are respectively a supersolution and a subsolution of any fully nonlinear equation in a proper cone of $\rn$.

\begin{proof}[Proof of Lemma \ref{alphlowbnd}]
According to \eqref{alphmon}, we may suppose that $\omega = \pi$. Due to \eqref{ellip}, we may assume that
\begin{equation*}
F(M,p,x) = \pucci(M) - \mu|x|^{-1} |p|.
\end{equation*}
Since this operator is rotationally invariant, we may also assume without loss of generality that $\xi = e_n : = (0,\ldots, 0,1)$.

For constants $\alpha,\kappa>0$ to be selected, consider the function
\begin{equation*}
\varphi(x) : = |x|^{-\alpha} \left( \exp(\kappa) - \exp\left( \kappa |x|^{-1} x_n \right) \right),
\end{equation*}
which is smooth in $\R^n \setminus \{ 0 \}$ and positive on $\cob{\omega}$. We will show that if $0 < \alpha < 1$ is very small and $\kappa > 0$ is very large, then for every $x \in \co\omega$,
\begin{equation} \label{albwts}
\pucci\left(D^2\varphi (x) \right) - \mu|x|^{-1}|D\varphi(x)| \geq 0.
\end{equation}
Routine calculations give
\begin{align*}
D\varphi(x) & = -\alpha|x|^{-\alpha-2} \left( e^\kappa - e^{\kappa |x|^{-1} x_n} \right) x + \kappa |x|^{-\alpha-3} e^{\kappa|x|^{-1}x_n} \left( |x|^2 e_n - x_nx \right), \\
D^2\varphi(x) & = \alpha |x|^{-\alpha-4} \left( e^\kappa - e^{\kappa |x|^{-1}x_n} \right) \left( (\alpha+2) x\otimes x - |x|^2 \iden\right) \\
& \quad +2 \kappa(\alpha+1) |x|^{-\alpha-5} e^{\kappa |x|^{-1}x_n} x \otimes (|x|^2e_n - x_nx) \\
& \quad - \kappa |x|^{-\alpha-5}x_n e^{\kappa|x|^{-1}x_n} \left( x \otimes x - |x|^2 \iden \right) \\
& \quad -\kappa^2 |x|^{-\alpha-6} e^{\kappa |x|^{-1} x_n } \left( |x|^2 e_n - x_n x\right) \otimes\left( |x|^2 e_n - x_n x\right).
\end{align*}
We easily estimate
\begin{equation*}
|D\varphi(x)| \leq \alpha |x|^{-\alpha-1} e^\kappa + \kappa |x|^{-\alpha-2} e^{\kappa|x|^{-1}x_n}\left( |x|^2 - x_n^2 \right)^{\frac12},
\end{equation*}
where we have used the equality $\left| |x|^2e_n - x_n x \right|^2 = |x|^2 (|x|^2-x_n^2)$.

To estimate $\pucci(D^2\varphi)$, we apply $\pucci$ to each term on the right side of the expression for $D^2\varphi(x)$ above, and sum them using the superlinearity property \eqref{puccineq}. We recall that for $a,b\in \R^n$ the nonzero eigenvalues of the symmetric matrix $a\otimes b + b \otimes a$ are $a\cdot b \pm |a||b|$, each with multiplicity $1$. Thus to calculate $\pucci$ applied to the second term in the expression for $D^2\varphi$, for example, we notice that $x \cdot (|x|^2e_n - x_nx) = 0$ and use the identity $\left| |x|^2e_n - x_n x \right|^2 = |x|^2 (|x|^2-x_n^2)$. With this in mind, and using \eqref{puccform} and \eqref{puccineq} and our estimate for $|D\varphi|$, after some work we obtain
\begin{equation} \label{tpl}
\begin{aligned}
\lefteqn{\pucci(D^2\varphi(x)) + \mu|x|^{-1} |D\varphi(x)| } \qquad & \\
& \geq -\alpha \left( \Lambda(\alpha+1) -\lambda(n-1) + \mu \right)  |x|^{-\alpha-2} e^\kappa \\
& \quad -\kappa \left( (\alpha+1)(\Lambda-\lambda) + \mu \right) |x|^{-\alpha-3} e^{\kappa|x|^{-1}x_n} (|x|^2-x_n^2)^{\frac{1}{2}} \\
& \quad + \kappa |x|^{-\alpha-5} e^{\kappa |x|^{-1} x_n} \pucci(-x_n(x\otimes x - |x|^2\iden)) \\
& \quad + \kappa^2 |x|^{-\alpha-4} e^{\kappa|x|^{-1} x_n} \lambda (|x|^2-x_n^2).
\end{aligned}
\end{equation}
We estimate the second term on the right side of \eqref{tpl} using Cauchy's inequality:
\begin{multline*}
-\kappa \left( (\alpha+1)(\Lambda-\lambda) + \mu \right) |x|^{-\alpha-3} e^{\kappa|x|^{-1}x_n} (|x|^2-x_n^2)^{\frac{1}{2}}\\
\geq -\textstyle\frac{1}{2} \kappa^2 |x|^{-\alpha-4} e^{\kappa|x|^{-1} x_n} \lambda (|x|^2-x_n^2) - \textstyle\frac{1}{2\lambda} \left((\alpha+1)(\Lambda-\lambda) + \mu \right)^2 |x|^{-\alpha-2} e^{\kappa |x|^{-1} x_n} .
\end{multline*}
Inserting this into \eqref{tpl} and imposing the requirement $\alpha \leq 1$ produces
\begin{equation} \label{tpl2}
\begin{aligned}
\lefteqn{\pucci(D^2\varphi(x)) + \mu|x|^{-1} |D\varphi(x)|} \qquad & \\
& \geq -\alpha C_1 |x|^{-\alpha-2} e^\kappa -  C_2 |x|^{-\alpha-2} e^{\kappa|x|^{-1} x_n}\\
& \quad + \kappa |x|^{-\alpha-5} e^{\kappa |x|^{-1} x_n} \pucci(-x_n(x\otimes x - |x|^2\iden)) \\
& \quad + \frac{1}{2} \kappa^2 |x|^{-\alpha-4} e^{\kappa|x|^{-1} x_n} \lambda (|x|^2-x_n^2).
\end{aligned}
\end{equation}
where $C_1 :=\max\{ 0 ,2\Lambda + \mu - \lambda(n-1) \}$ and $C_2 : = (2\Lambda+\mu)^2/(2\lambda)$. The terms on the right side of \eqref{tpl2} must be estimated in the regions $\{ x_n \leq 0\}$ and $\{ 0\leq x_n \leq \sigma|x| \}$ separately. First, in the region $\{ x_n \leq 0\}$, we have
\begin{equation*}
\pucci(-x_n ( x\otimes x - |x|^2 \iden)) = \lambda(n-1) |x|^2 |x_n|.
\end{equation*}
Then in the region $\{ x_n \leq 0 \}$ we obtain
\begin{multline*}
\pucci(D^2\varphi(x)) - \mu|x|^{-1} |D\varphi(x)| \geq -\alpha C_1 |x|^{-\alpha-2}e^\kappa \\
 + |x|^{-\alpha-2} e^{\kappa|x|^{-1} x_n} \left( \kappa \lambda (n-1) |x|^{-1} x_n + \textstyle\frac{1}{2} \kappa^2 \lambda |x|^{-2} (|x|^2-x_n^2)  - C_2 \right).
\end{multline*}
Using  $\max\{ |x|^{-1} x_n ,1-x_n^2 / |x|^2\} \geq 1/2$ we obtain
\begin{equation}
\pucci(D^2\varphi(x) ) - \mu|x|^{-1} |D\varphi(x)| \geq  |x|^{-\alpha-2}  \left( -\alpha C_1e^\kappa + e^{-\kappa} \left( \textstyle\frac{1}{4}\lambda\kappa - C_2 \right) \right).
\end{equation}
provided $\kappa \geq 1$. We now impose the requirement $\kappa \geq 4(C_2 +1 )/ \lambda$ as well as $C_1\alpha \leq e^{-2\kappa}$, so that we have \eqref{albwts} in the region $\{ x_n \leq 0 \}$.

In $\{ 0 \leq x_n \leq \sigma |x| \}$, the factor in the middle term on the right side of \eqref{tpl2} is
\begin{equation*}
\pucci(-x_n ( x\otimes x - |x|^2 \iden)) = - \Lambda(n-1) |x|^2 x_n,
\end{equation*}
and we have, using also that $x_n^2 \leq\sigma^2|x|^2$,
\begin{multline*}
\pucci(D^2\varphi(x)) -\mu |x|^{-1} |D\varphi(x)| \geq -\alpha C_1 |x|^{-\alpha-2} e^\kappa  \\
+ |x|^{-\alpha-2} e^{\kappa|x|^{-1} x_n} \left( -\kappa \Lambda (n-1) + \textstyle\frac{1}{2} \kappa^2 \lambda (1-\sigma^2)  - C_2 \right).
\end{multline*}
If we require that $\kappa \geq 1 + C_2 + (2\Lambda(n-1)+1)/(\lambda(1-\sigma^2))$, then the term in the parentheses is at least 1. We then obtain \eqref{albwts} in the region $\{ 0 \leq x_n \leq \sigma|x|\}$ provided that $C_1\alpha\leq e^{-\kappa}$.

We have shown that \eqref{albwts} holds if we set
\begin{align*}
& \kappa : = 1+ C_2 + (2\Lambda(n-1)+1)/(\lambda(1-\sigma^2)) + 4(C_2+1) / \lambda, \\
& \alpha : = \min\{ 1 , e^{-2\kappa} / C_1 \}.
\end{align*}
It follows that $\alpha^+(\pucci(D^2\cdot) - \mu|x|^{-1} |D\cdot|,\omega) \geq \alpha$.
\end{proof}

\begin{proof}[Proof of Lemma \ref{alphuppbnd}]
It suffices to prove the lemma for the operator $F(M,p,x) = \Pucci(M) + \mu|x|^{-1} |p|$. Since $F$ is rotationally invariant, we may also suppose without loss of generality that $\xi = e_n := (0,\ldots,0,1)$. By \eqref{alphmon}, we may assume that $\omega = \pi$.

Define the functions
\begin{equation*}
w(x) : = |x|^{-\alpha-2} x_n^2 - \sigma^2 |x|^{-\alpha} \quad \mbox{and} \quad \varphi(x) : = \frac{1}{2}(w(x))^2.
\end{equation*}
It is clear that $\varphi$ is smooth on $\R^n \setminus \{ 0 \}$, $\varphi \in \hc{2\alpha}$, $\varphi > 0$ in $\co{\omega}$ and $\varphi$ vanishes on $\partial \co{\omega} \setminus \{ 0 \}$. We claim that for sufficiently large $\alpha > 0$, the function $\varphi$ satisfies
\begin{equation}\label{tupw}
\Pucci(D^2\varphi) + \mu|x|^{-1} |D\varphi(x)|< 0 \quad \mbox{in} \ \co\omega.
\end{equation}
Assuming for a moment that \eqref{tupw} is satisfied for all $\alpha\geq \bar \alpha > 0$, with $\bar\alpha$ depending on the appropriate constants, let us complete the proof by showing that $\alpha^+(F,\omega) \leq 2 \bar \alpha$. If on the contrary $\alpha^+(F,\omega) > 2\bar\alpha$, then there exists $2\bar\alpha < \alpha \leq \alpha^+(F,\omega)$  and a function $v \in \hc{\alpha}$ satisfying $F(D^2v,Dv,x) \geq 0$ and $v> 0$ in $\co\omega$. Thus according to Proposition~\ref{hcp}, $\varphi \equiv tv$ for some $t> 0$. This is impossible, since $\varphi$ is a \emph{strict} subsolution of \eqref{tupw}.

Commencing with the demonstration of \eqref{tupw}, routine calculations give
\begin{align*}
Dw(x) & = -\alpha |x|^{-2} w(x) x + 2 |x|^{-\alpha-4} x_n (  |x|^2 e_n - x_nx) , \\
D^2w(x) & = \alpha |x|^{-4} w(x) \left( (\alpha+2) x \otimes x - |x|^2\iden \right) - 2|x|^{-\alpha-4} \left( x_n^2\iden -|x|^2e_n\otimes e_n \right)\\
& \quad -4(\alpha+2) |x|^{-\alpha-6} x_n x\otimes (|x|^2e_n - x_nx).
\end{align*}
Further computations produce
\begin{align*}
D\varphi(x) & = w(x) Dw(x) \\
& = \alpha |x|^{-2} w(x)^2 x + 2 |x|^{-\alpha-4} w(x) x_n \left( |x|^2e_n - x_nx \right),\\
D^2\varphi(x) & = w(x) D^2w(x) + Dw(x) \otimes Dw(x) \\
& = 2\alpha (\alpha+1) |x|^{-4} w(x)^2 x\otimes x - \alpha |x|^{-2} w(x)^2 \iden - 2|x|^{-\alpha-4} x_n^2 w(x) \iden \\
& \quad -8|x|^{-\alpha-6} x_n w(x)  x\otimes (|x|^2e_n - x_nx) + 2|x|^{-\alpha-2} w(x) e_n \otimes e_n \\
& \quad +4|x|^{-2\alpha-8} x_n^2\left( |x|^2e_n - x_nx\right)\otimes \left( |x|^2e_n - x_nx\right).
\end{align*}
We estimate, using once again the identity $\left| |x|^2e_n - x_n x \right|^2 = |x|^2( |x|^2 - x_n^2)$,
\begin{equation*}
|D\varphi(x)| \leq \alpha |x|^{-1} w(x)^2 + 2|x|^{-\alpha-3} w(x) |x_n| \left( |x|^2-x_n^2\right)^{\frac12}, \quad x\in \co\omega.
\end{equation*}
We recall again that the matrix $a\otimes b + b\otimes a$ has two nonzero eigenvalues, namely $a\cdot b \pm |a||b|$. Using $x\cdot (|x|^2e_n - x_nx) =0$ and $\left| |x|^2e_n - x_n x \right|^2 = |x|^2( |x|^2 - x_n^2)$, the sublinearity and homogeneity of $\Pucci$, that $w > 0$ in $\co\omega$, and our estimate for $|D\varphi|$ above, we have after some calculation
\begin{equation} \label{tpu}
\begin{aligned}
\lefteqn{\Pucci(D^2\varphi(x)) + \mu|x|^{-1} |D\varphi(x)|} \qquad & \\
& \leq -\alpha \left( 2\lambda \alpha - \mu - (n-1)\Lambda \right)  |x|^{-2} w(x)^2 + 2n\Lambda |x|^{-\alpha-4} x_n^2 w(x) \\
& \quad + (4\Lambda+2\mu) |x|^{-\alpha-4} |x_n| \left(|x|^2-x_n^2\right)^{\frac12} w(x)  - 4\lambda |x|^{-2\alpha-6}x_n^2(|x|^2-x_n^2)
\end{aligned}
\end{equation}
for $x\in \co\omega$. We estimate the two middle terms using Cauchy's inequality:
\begin{equation*}
2n\Lambda |x|^{-\alpha-4} x_n^2w(x) \leq \alpha n^2\Lambda^2\lambda^{-1} |x|^{-2} w(x)^2 + \alpha^{-1} \lambda |x|^{-2\alpha-6} x_n^4,
\end{equation*}
\begin{multline*}
(4\Lambda +2\mu) |x|^{-\alpha-4} |x_n|\left(|x|^2-x_n^2\right)^{\frac12} w(x) \\
\leq \frac{1}{2}(2\Lambda+\mu)^2\lambda^{-1}|x|^{-2} w(x)^2 + 2\lambda |x|^{-2\alpha-6} x_n^2 (|x|^2-x_n^2),
\end{multline*}
and inserting these into \eqref{tpu} gives
\begin{multline*}
\Pucci(D^2\varphi(x)) + \mu|x|^{-1} |D\varphi(x)| \leq -\alpha C_3 |x|^{-2} w(x)^2 \\
-2\lambda |x|^{-2\alpha-6} x_n^2(|x|^2-x_n^2) + \alpha^{-1}\lambda |x|^{2\alpha-6} x_n^4
\end{multline*}
where $C_3: = 2\lambda \alpha - \mu- (n-1) \Lambda - n^2 \Lambda^2\lambda^{-1} - \textstyle\frac12(2\Lambda+\mu)^2 \lambda^{-1}\alpha^{-1}$. If we require
\begin{equation} \label{tupa1}
\alpha \geq 1 + (2\lambda)^{-1} \left( \mu + (n-1)\Lambda + n^2\Lambda^2 \lambda^{-1} + \textstyle\frac12 (2\Lambda+\mu)^2 \lambda^{-1} \right),
\end{equation}
then $C_3\geq 1$, and we obtain
\begin{multline} \label{tpu2}
\Pucci(D^2\varphi(x)) + \mu|x|^{-1} |D\varphi(x)|  \\ \leq -\alpha |x|^{-2} w(x)^2-2\lambda |x|^{-2\alpha-6} x_n^2(|x|^2-x_n^2) + \alpha^{-1}\lambda |x|^{-2\alpha-6} x_n^4.
\end{multline}
We may rewrite \eqref{tpu2}, using the definition of $w$ and rearranging terms, as
\begin{multline} \label{tpu3}
\Pucci(D^2\varphi(x)) + \mu|x|^{-1} |D\varphi(x)| \\ \leq -|x|^{-2\alpha-6} \left( \alpha \sigma^4 |x|^4 - (2\alpha\sigma^2-2\lambda) x_n^2|x|^2 + (\alpha-2\lambda - \alpha^{-1}\lambda) x_n^4 \right).
\end{multline}
By Cauchy's inequality,
\begin{equation*}
(2\alpha\sigma^2 - 2\lambda) x_n^2 |x|^2 \leq \alpha\sigma^4|x|^4  + \frac{(2\alpha\sigma^2-2\lambda)^2}{4\alpha\sigma^4} x_n^4,
\end{equation*}
and therefore the term in the parenthesis on the right side of \eqref{tpu3} is positive provided that
\begin{equation*}
(2\alpha\sigma^2 -2\lambda)^2 - 4\alpha\sigma^4(\alpha-2\lambda-\alpha^{-1}\lambda) < 0,
\end{equation*}
which is equivalent to
\begin{equation} \label{tupa2}
\alpha > \frac{\lambda + \sigma^4}{2\sigma^2(1-\sigma^2)}.
\end{equation}
We have shown that \eqref{tupw} is satisfied for $\alpha$ satisfying both \eqref{tupa1} and \eqref{tupa2}. It follows that
\begin{equation*}
\alpha^+(F,\omega) \leq 2 + \lambda^{-1} \left( \mu + (n-1)\Lambda + n^2\Lambda^2 \lambda^{-1} + \textstyle\frac12 (2\Lambda+\mu)^2 \lambda^{-1} \right) + \frac{\lambda + \sigma^4}{\sigma^2(1-\sigma^2)}. \qedhere
\end{equation*}
\end{proof}

\subsection*{Acknowledgements}
The research for this article was conducted in part while the first and third authors were visitors of Le Centre d'analyse et de math\'ematique sociales (CAMS) in Paris. The first author was partially supported by a CNRS visiting position and NSF Grant DMS-1004645. The third author was partially supported by NSF grant DMS-1004595.
\bibliographystyle{plain}
\bibliography{conesing}

\end{document}